\documentclass[10pt]{amsart}
\usepackage{graphicx}
\usepackage{amssymb, amsthm, amsmath}
\usepackage{wasysym}
\usepackage{amscd}
\usepackage{endnotes}
\usepackage[active]{srcltx}
\usepackage[stable]{footmisc}
\usepackage{float}
\usepackage{caption}
\usepackage[matrix, arrow,all,cmtip,color]{xy}
\usepackage{hyperref}

\restylefloat{figure}

\usepackage{headers-note}
\usepackage{epigraph}
\usepackage[margin=0.85in]{geometry}
 
\makeatletter
\usepackage{amssymb, stmaryrd}
\makeatother

\newcommand{\CA}{\mathcal A}
\newcommand{\CB}{\mathcal B}

\newcommand{\CF}{{\mathcal F}}

\newcommand{\CN}{{\mathcal N}}

\newcommand{\CX}{\mathcal X} 
\newcommand{\CY}{\mathcal Y}
\newcommand{\CV}{\mathcal V}
\newcommand{\CZ}{\mathcal Z}

\newcommand{\FC}{\mathfrak C}

\def\cov{\mathrm{cov}}
\def\pp{\mathrm{pp}}

\def\Bbb{\mathbb}

\def\L{{\Bbb L}}

\def\acc{\textsf{acc\,}}

\def\StNaamen{\textrm{StNaamen}}
\def\QtNaamen{\textrm{QtNaamen}}
\def\QtN{\textrm{QtNaamen}}
\def\Qt{\textrm{QtNaamen}}

\def\cof{{\mathrm{cof}\,}}
\def\rightc{\xrightarrow{(c)}}

\def\rightf{\xrightarrow{(f)}}

\def\rightw{\xrightarrow{(w)}}
\def\leftw{\xleftarrow{(w)}}
\def\rightwc{\xrightarrow{(wc)}}

\def\rightwf{\xrightarrow{(wf)}}
\def\rightcwf{\xrightarrow{(cwf)}}

\def\0{\varnothing}
\def\rtt{\rightthreetimes}
\def\lra{\longrightarrow}

\def\rra{\Rightlongarrow}

\def\ilim{{\raisebox{0pt}{$\bigcirc$}} \kern -0.31cm \hbox{$\Yleft$}}

\def\aB{\aleph_{\beta}}
\def\ab*{\aleph^*_{\beta}}
\def\aa{\aleph_{\alpha}}
\def\pb{\perp_\beta}

\def\qtk{\StNaamen_{\kappa}} \def\qtkappa{\qtk}
\def\qtkk{\StNaamen^+_{\kappa}}
\def\ukappa{\bigcup_{<\kappa}} \def\Qtk{\QtNaamen_{\kappa}}
\def\Qtkk{\QtNaamen^+_{\kappa}}

\DeclareMathAlphabet{\mathpzc}{OT1}{pzc}{m}{it}

\newcommand\parno{\stepcounter{subsubsection}\paragraph}
\def\C{{\Bbb C}}

\def\L{\mathbb L}
\def\LL{{\Bbb L}}
\def\Lc{{\Bbb L_c}}

\CompileMatrices

\author[M. Gavrilovich]{Misha Gavrilovich}
\thanks{The first author was partially supported by a MODNET (European Commission
Research Training Network) grant and by the Skirball foundation as a post-doctoral fellow at Ben Gurion University.}
\address{Department of mathematics\\
Ben Gurion University of the Negev\\
Be'er Sheva,\\
Israel}
 \email{gavrilovich@gmail.com}
\author[A. HASSON]{Assaf Hasson}
\thanks{The second author was partially supported by GIF grant No. 2266/2010 and by ISF grant No. 1156/10}
\address{Department of mathematics\\
Ben Gurion University of the Negev\\
Be'er Sheva,\\
Israel} \email{hassonas@math.bgu.ac.il}

\title{Exercices de style: A homotopy theory for set theory II}

\begin{document}

\begin{abstract}
	This is the second part of a work initiated in \cite{GaHa}, where we
	constructed a model category, $\Qt$, for set theory. In the present paper we
	use this model category to introduce homotopy-theoretic intuitions to set
	theory. Our main observation is that the homotopy invariant version of
	cardinality is the covering number of Shelah's PCF theory,
	and that other combinatorial objects, such as
	Shelah's revised power function -
	the cardinal function featuring in  Shelah's revised GCH theorem -
	can be obtained using similar tools.  We include a small ``dictionary'' for
	set theory in $\QtNaamen$, hoping it will help in finding more meaningful
	homotopy-theoretic intuitions in set theory.
\end{abstract}

\maketitle


\epigraph{
Every man is apt to form his notions of things difficult to be apprehended,
or less familiar, from their analogy to things which are more familiar.
Thus, if a man bred to the seafaring life, ...
should take it into his head to
philosophize concerning the faculties of the mind, it cannot be doubted,
but he would draw his notions from the fabric of the ship, and would find
in the mind, sails, masts, rudder, and compass.}
{                -- Thomas Reid, "An Inquiry into the Human Mind", 1764
}

\section{introduction}
This is the second part of the paper \cite{GaHa}, and continues the work initiated therein. In the first part of this paper we constructed a model category, $\QtNaamen$, for set theory. Loosely speaking, $\QtNaamen$ can be thought of as the simplest model category for set theory modelling the notions of finiteness, countability and (infinite) equi-cardinality. From the purely category theoretic point of view $\QtNaamen$ is extremely simple (arrows are unique whenever they exist, so - e.g. - all diagrams commute), but as a model category the picture is slightly more complicated. On the one hand, most basic tools of model categories (such as the loop and suspension functors) degenerate in $\QtNaamen$, but - on the other hand - as a model category $\QtNaamen$ does not seem to be such a trivial object (and the homotopy category associated with it is - at least to us - a new set theoretic object). 

From the homotopy theoretic point of view, given the axioms of model categories and the set theoretic notions to be modelled the construction of $\QtNaamen$ is almost automatic (this is one of the main themes of \cite{GaHa}). Therefore, from that viewpoint $\QtNaamen$ should be an almost unavoidable (though somewhat degenerate) object. But, as far as we were able to ascertain, $\QtNaamen$ (or any close relative thereof) is not known (under the appropriate translation to set theoretic language, of course) to set theorists. On the face of it, it could be that the reason $\QtNaamen$ was not discovered by neither homotopy theorists nor set theorists is that it is too degenerate to be of interest. The aim of this paper is to show that this is, maybe,  not entirely true. In the main result of the present paper we show that Shelah's covering numbers - one of the main objects of interest in PCF theory - discovered  a century or so after Cantor's introduction of the notions of countability and cardinality, cannot be missed if one tries to study these notions from the homotopy theoretic point of view. Technically, we prove: 

\begin{thm}\label{mainintro}
 Let $\lambda$ be any cardinal. Then 
\[
 \Lc\card(\{\lambda\})=\cov(\lambda,\aleph_1,\aleph_1,2).
\]
where $\Lc\card$ is the cofibrantly replaced  left derived functor of the cardinality function (not functor (!)) $\card:QtNaamen\to On^\top$. 
\end{thm}
The proof of Theorem \ref{mainintro} is, essentially, a triviality, but its formulation - at least for those not fluent in model category jargon - is far from obvious. The main part of this paper is dedicated to explaining the formulation of Theorem \ref{mainintro}, and explaining - given the model category $\QtNaamen$ and the cardinality function $\card:QtNaamen\to On^\top$ (where $On^\top$ is the class of ordinal augmented by a terminal object) - how to obtain the functor $\Lc\card$. We then show how other covering numbers (such as Shelah's revised power function) can be recovered. This and similar constructions are discussed in Section \ref{other}.

It is intriguing that Shelah, in his book on Cardinal arithmetic, \cite{ShCard}, and Kojman, in his survey of Shelah's PCF theory, \cite{KojABC}, use algebraic topology as an analogy to explain the ideology and the usefulness of this theory, Kojman writes, rather directly that ''This approach to cardinal arithmetic can be thought of as 'algebraic set theory' in analogy to algebraic topology" and Shelah, more by way of example mentions that: "...  for a polyhedron $v$ (number of vertices),  $e$ (number of edges) and $f$ (number of faces) are natural measures, whereas $e + v + f$ is not, but from deeper point of view [the homotopy-invariant Euler characteristic] $v - e + f$ runs deeper than all...".  Theorem \ref{mainintro} and its variants can be viewed as consolidating this analogy: they show that (some constructs of) PCF theory have an actual interpretation in terms of algebraic topology.  But - at this stage - it is not clear whether this can be pushed much further, whether this connection with algebraic topology runs any deeper. 

Of course, the covering numbers are not the only set theoretic notions that one can recover in $\QtN$. In \cite{GaHa} we saw that $\QtNaamen$ models finiteness, countability and equi-cardinality (at least to some extent). In the present paper we slightly enlarge the set theoretic dictionary of $\QtNaamen$, giving some natural examples and non-examples (of set theoretic notions that $\QtN$ cannot capture - e.g., the power set of a set). Notions such as a cardinal being measurable (Lemma \ref{measurable}) and intriguing possible connections with Jensen's covering lemma are discussed in Section \ref{other}. 

As already mentioned, this is the second part of \cite{GaHa}. We expect readers of this paper to be familiar with the terminology and notation of \cite{GaHa}, but for ease of reference we dedicate Section 2 to a a concise rendering of the main definitions and notational conventions. 
 In Section 3 - as a warm up - we discuss various examples on how to use the model category $\QtNaamen$ in order to describe some basic notions of set theory, and describe (some of) its limitations. The statement and proof of Theorem \ref{mainintro} are given in Section 4: we explain the notions of derived functors and how to compute them in posetal categories. The last sub-section of Section 4 is dedicated to a brief overview of possible variants. We conclude the paper with some ideas for further investigation, emanating mainly form problems we identified in our construction: can we overcome the dependence of the derived functor of, say, cardinality on the choice of the model category (among equivalent model categories), can we find analogues for homotopy theory constructs (homotopy groups, long exact sequences etc.) in $\Qt$ despite of it being "degenerate", can we actually prove set theoretic statements using $\Qt$ (and the family of model structures $\Qt_\kappa$) 
 and not only recover known concepts and definitions?

\begin{rem}[Set-theoretic foundations]
It is often the case when working with categories that the category is \emph{large}, namely that the objects and morphisms do not form a set, but rather a proper class. In the present work the objects of $\QtNaamen$ themselves are proper classes, and $\Ob\QtNaamen$ is the collection of all classes. To avoid paradoxes one has to be careful, so some words concerning foundational issues may be in place. 

There are many standard solutions for situations as described in the previous paragraph (most of them the authors are not familiar enough to say much about), and as we are using very little set theory, we believe that any of them could suit us with essentially no effect on the results. Probably, the  easiest way to avoid foundational issues is (assuming the consistency of ZFC, of course) to choose and fix a transitive set-sized model $\underline\CV=(\CV,\in)$
of ZFC, and consider all constructions as taking place within $\underline\CV$: read below a {\em class} 
as {\em a $V$-definable subset of $\CV$} (namely, the objects of our category are elements of $\mathbb P(V)$).  This proves the consistency of our  construction (relative to the consistency of ZFC). Stronger assumptions (e.g. large cardinal assumptions) could provide us with models whose notions of subset, ordinals, cardinals etc. coincides with the the corresponding notions in the "real" universe.  We remark, moreover, that our construction seems to fit quite easily in set theories equi-consistent with ZFC, such as NBG.
\end{rem}

\section{Definitions of c-w-f arrows, notation and the construction of the model category.}

We assume the reader familiar with the notation and terminology of \cite{GaHa}, but for ease of reference we dedicate the present section to a concise rendering of the main definitions, notational conventions and results
of \cite{GaHa}. 

Recall that a category $\FC$ is a pair $(\Ob \FC, \Mor\FC)$ of \emph{objects} and \emph{morphisms} (or \emph{arrows}) each carrying its own notion of equality. The arrows of a category can be composed whenever the composition makes sense (i.e. the arrows $X\lra Y$ and $Y'\lra Z$ can be composed whenever the objects $Y$ and $Y'$ are equal to produce and arrow $X\lra Z$). It is also required that to any object $X$ there is a special morphism $\id_X$, neutral with respect to left and right composition. Given a category $\FC$ it is often convenient (and we do it quite often in the present paper) to represent data in $\FC$ by means of a directed graph (possibly with loops and multiple edges between two vertices) whose nodes are objects and whose edges are morphisms. Such a diagram is \emph{commutative} if the composition of arrows along any path in the graph depends only on the starting point and the ending point of the path, but not on the path itself. 

A \emph{labelled category} is a category where to each arrow is associated a (possibly empty) set of labels of a set $S$ of labels. We require that for any label $s\in S$ the collection of all $s$-labelled arrows is itself a category. Namely, given a category $\FC$ the collection 
\[
\Ob_s \CF:=\{X\in \Ob\FC: \exists Y\in \Ob \FC: Y\xrightarrow{(s)} X\lor X\xrightarrow{(s)} Y\}
\]
with the collection $\Mor_s\FC$ of all $s$-labelled arrows of $\FC$ is a category. Observe that in a labelled category the identity morphisms must carry all labels. 

Before we explain what is a model category, it will be convenient to remind that given a category $\FC$ and arrows 
\begin{figure}[H]
\centerline{
\xymatrix @R=2pc @C=2pc{
X \ar[r] \ar[d] 
& W \ar[d] \\
Y \ar[r]
& Z }}
\caption{}\label{square}
\end{figure}
The arrow $X\lra Y$ lifts with respect to the arrow $W\lra Z$ if 
for every commutative diagram as in Figure~\ref{square}, 
there exists an arrow $Y\lra W$ making the resulting diagram commute. We denote this property by $X\lra Y\rtt W\lra Z$ and say that $X\lra Y$ left lifts with respect to $W\lra Z$ (or that $W\lra Z$ right lifts with respect to $X\lra Y$). Note that the notation $X\lra Y\rtt W\lra Z$ implicitly implies that  $X\lra Y$ and $W\lra Z$) but not necessarily that $X\lra W$ and $Y\lra Z$. In
 particular, the lifting property $X\lra Y\rtt W\lra Z$ holds (vacuously) if there does
 not exist an arrow $X\lra W$ or if there does not exist an arrow $Y\lra Z$. 

A model category is a $\{c,f,w\}$-labelled category, $\FC$, satisfying the following axioms: 

\begin{description}
 \item[(M0)] As a category $\FC$ is closed under (finite) direct and inverse limits. 
\item[(M1)] $(wc)\rtt (f)$ and $(c)\rtt (wf)$ (i.e., any appropriately labelled diagram as in Figure~\ref{square} has the lifting property).  
\item[(M2)] For any arrow $X\lra Y$ there are objects $X_{(wc)}$ and $X_{(wf)}$ such that $X\rightwc X_{(wc)}\rightf Y$ and $X\rightc X_{(wf)} \rightwf Y$ making the resulting diagrams commute. 
\item[(M3)] This is the axiom asserting that a model category is a labelled category. 
\item[(M4)] The pushforward of an arrow labelled (wc) and the pullback of an arrow labelled (wf) are both labelled (w).
\item[(M5)] Given a triangle $X\lra Y\lra Z\lra X$, if any two of the arrows are labelled $(w)$ so is the third. 
 \end{description}

If, in addition, the model category satisfies the requirement that any two of the labels determine the third, the model category is called \emph{closed}. 

\subsection{The model category $\QtNaamen$}
The model category $\QtNaamen$ whose construction is the main concern of \cite{GaHa} can be, roughly, thought of as the simplest model category modelling the notions of finiteness, countability and equi-cardinality. Below we give a combinatorial rendering of the $\{c,w,f\}$-labelling of our category. From this point of view, the construction may seem somewhat mysterious - so we start with a more ``geometric'' overview of the construction. 

We start with the category $\rm{Sets^\subseteq}$, whose objects are sets, and whose arrows are inclusions: 
\begin{description}
 \item[$A\lra_0 B$] $A$ and $B$ are sets and $A\subseteq B$. 
\item[$A\rightwc_0 B$] $A\lra_0 B$ and $B\setminus A$ is finite. 
\item[$A\rightc_0 B$] $A\lra_0 B$ and $\card(A)+\aleph_0=\card(B)+\aleph_0$. 
 \end{description}
This category does not satisfy (M0) as it does not have a terminal object, so we add one formally, $\top$. But now the arrow $\0\lra \top$ does not satisfy (M2): it is not too hard to see that $\0_{(wc)}$ should be the ``direct limit'' of all finite sets - which can be identified with the (proper) class of all finite sets. So we replace the category we are working with. The objects, $\Ob\StNaamen$, are all classes of sets and the morphisms, $\Mor\StNaamen$ are given by: 

\begin{description}
 \item[$X\lra Y$] For all $x\in X$ there exists $y\in Y$ such that $x\lra_0 y$. 
\end{description}

We can naturally identify $\rm{Sets^\subseteq}$ with a (full) sub-category of $\StNaamen$, inducing a labelling  (by the labels (c) and (wc)) of a class of arrows in $\Mor \StNaamen$. The full labelling on $\StNaamen$ is the one generated by this labelling (namely, the most economical labelling on $\StNaamen$ respecting the labelling of $\rm{Sets^\subseteq}$ and satisfying (M2)). More precisely: 

\begin{description}
\item[$X\rightf Y$] $X\lra Y$ and $(wc)_0 \rtt X\lra Y$. 
\item[$X\rightwf Y$] $X\lra Y$ and $(c)_0 \rtt X\lra Y$. 
\item[$X\rightc Y$] $X\lra Y$ and $X\lra Y\rtt (wf)$. 
\item[$X\rightwc Y$] $X\lra Y$ and $X\lra Y\rtt (f)$. 
\item[$X\rightw Y$] if $X \lra Y$ and there exists $Z$ such that $X\rightwc Z\rightwf Y$. 
\end{description}
 
The above labelling has a combinatorial interpretation (see \cite[Proposition 16]{GaHa}): 

\begin{prp}\label{combinatorics}
The set theoretic interpretation of the last definition is: 

\bi
\item[(f)]an arrow $\CA\lra \CB$ is labelled (f) if and only if
for every $A\in \CA\cup\{\emptyset\}$, $B\in \CB$ and a finite
subset $\{b_1,\dots ,b_n\}\subseteq  B$ there exists $A'\in \CA\cup\{\emptyset\}$ 
such that $(A\cap B)\cup\{b_1,...,b_n\}\subseteq A'$.

\item[(wf)]an arrow $\CA\lra \CB$ is labelled (wf) if and only if
for every $A\in \CA\cup\{\emptyset\}$, $B\in \CB$ and subset $B'\subseteq B$ such that $\card B'\leq \card (A\cap B)+\aleph_0$,
 there exists $A'\in \CA\cup\{\emptyset\}$
such that $B'\subseteq A'$. 

\item[(wc)] an arrow $\CA\lra \CB$ is labelled (wc) if and only if
every $B\in \CB$ is contained, up to finitely many elements, in some $A\in \CA\cup\{\emptyset\}$ (i.e. $B\setminus A$ is finite for some $A\in \CA\cup\{\emptyset\}$).

\item[(c)]  an arrow $\CA\lra \CB$ is labelled (c) if and only if
for every $\{B\}\lra \CB$ there exists $A\in \CA\cup\{\emptyset\}$ such that $A\xrightarrow{\,\CB\,} B$, where we define 
$\CA\xrightarrow{\,\CB\,} B$ if there exist $n\in \mathbb N$ and $\{B_0,\dots B_n\}\lra \CB$ such that: 
\begin{enumerate}
 \item $\card(A\cap B_0)+\aleph_0=\card B_0 +\aleph_0$ , 
\item $\card(B_i\cap B_{i+1})+\aleph_0=\card B_{i+1}+\aleph_0$ for all $0\le i < n$, and 
\item $B=B_n$. 
\end{enumerate}
 
\item[(w)]an arrow $\CA\lra \CB$ is labelled (w) if and only if
for every $A\in \CA$, $B\in \CB$ and
subset $B'\subseteq B$ such that $\card B'\leq \card (A\cap B)+\aleph_0$,
 there exists $A'\in \CA$
such that $B'$ is contained in $A'$ up to finitely many elements.
\ei
\end{prp}

It turns out, using Proposition \ref{combinatorics}, that $\StNaamen$ with the above labelling satisfies (M0)-(M4). Moreover, the composition of two weak equivalences (i.e., the composition of two $(w)$-labelled arrows) is a weak equivalence (a fact that we will use freely).  But $\StNaamen$ does not satisfy the full axiom (M5). The following counter example is relatively easy to come by:

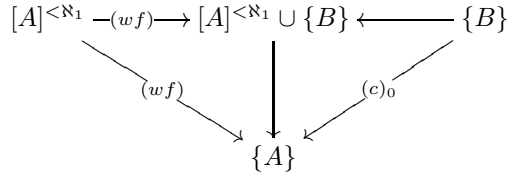
\begin{figure}[H]
\centerline{
 \xymatrix  @R=3pc @C=3pc{
[A]^{<\aleph_1} \ar[r]|-(0.4){(wf)} \ar[dr]|-{(wf)} & \ar[d][A]^{<\aleph_1} \cup \{B\}  & \{B\} \ar[l] \ar[dl]|-{(c)_0}\\
& \{A\} &
}}
\caption{Let $A,B$ be uncountable sets of the same cardinality and $[A]^{<\aleph_1}$ be the set of all countable subsets of $A$. Then the labelling of the arrows is as above, but $[A]^{<\aleph_1}\cup \{B\}\lra \{A\}$ will not, in general, be labelled (w).}
\end{figure}

\noindent This problem is addressed by restricting to a sub-category: 
\begin{definition}
 Say that an object $\CX$ in $\StNaamen$ is \emph{cute} if it satisfies the following diagram: 
\begin{figure}[H]
\centerline{
  \xymatrix  @R=3pc @C=3pc{
\ar@/^1pc/[rr] \ar[d]|-{(c)_0} & \ar[r] \ar[dl]|-{(wf)} & \CX \\
  \ar@{.>}[urr] 
}}
\caption{The diagram reads: for every commutative diagram of solid arrows as above, there exists a dashed arrow such that the resulting diagram commutes.}
We let $\QtNaamen$ be the full labelled sub-category of cute objects of $\StNaamen$. 
\end{figure}\label{Qt}
\end{definition}

The main result of \cite{GaHa} is: 
\begin{thm}
 The labelled category $\QtNaamen$ is a model category. The labelling is generated by co-fibrations between singletons. 
\end{thm}

We remark that the category $\QtNaamen$ has limits and co-limits of arbitrary (not necessarily small) \emph{well-defined} 
diagrams [A\S3.1, Remark 9]; the notion of of ``well-defined'' depends on the set theoretic  foundations used. E.g., if the objects of $\Qt$ are \emph{cute} classes, then the limits of any class of classes (i.e. any uniformly definable collection of classes) exist.

\section{The expressive power of $\QtNaamen$.} 

As already mentioned several times before, $\QtNaamen$ is a very simple model category. Intuitively, $\QtNaamen$ should be much simpler than set theory. To formulate this intuition somewhat more precisely, we observe that: 

\begin{lem}
	Let $\bf V$ be the universe of set theory and $\sigma$ a bijective class function on $\bf V$. For a class $X\subseteq \bf V$ let $\tilde \sigma(X)=\{\{\sigma(a):a\in x\}:x\in X\}$. 
	Then $\tilde \sigma: \QtNaamen\lra \QtNaamen$ is a bijective functor on $\QtNaamen$. Moreover, $\tilde \sigma$ preserves the model structure of $\QtNaamen$. 
\end{lem}
\begin{proof}
 Because $\sigma$ is a class function, if $X$ is a class so is $\tilde\sigma(X)$ (hence, $\tilde \sigma$ is indeed a functor from $\QtNaamen$ to itself). The only non-trivial part is that $\tilde \sigma$ preserves the model structure, which is an immediate corollary of Proposition \ref{combinatorics}. 
\end{proof}

Observe that in ZFC, given a set $S$, any $\sigma\in \rm{Sym}(S)$ extends to a class-bijection of $\bf V$  by setting $\sigma(x)=x$ for $x\not\in S$. Therefore, the last lemma proves that the model structure on $\QtNaamen$, while it must recognize the subset relation, does not respect - in a strong sense, the membership relation. For example, for a set $X$ the set theoretic operation $X\mapsto \{X\}$ is not respected by $\QtNaamen$, as can be inferred from the existence of  an automorphism exchanging $\{\{\0\}\}$ with $\{\{a\}\}$ (for any set $a$).

Even more trivially, since for any set $S$ we have $\{S\}\longleftrightarrow \mathbb P(S)$, we see that $\QtNaamen$ cannot distinguish $\{S\}$ from the power set of $S$. Thus, despite of the fact that $\QtNaamen$ was constructed specifically to model the notion of equi-cardinality, it does so with limited success. Moreover, the notion of a set being a singleton is also a notion unknown to $\QtNaamen$, as shows the above example. 

In order to extract meaningful information from the model category $\QtNaamen$ we can - as is standard in mathematics - beside studying the structure of $\QtNaamen$ itself, study functors (and other ``natural'' set theoretic functions) from $\QtNaamen$ to other categories, and vice versa. The next, section, for  example, is dedicated to the study of the cardinality function (not functor) $\card: \QtNaamen \lra On^\top$. In the present section we perform easier computations, showing that by imposing a little extra ``natural'' set theoretic structure on $\QtNaamen$, more information can be obtained.

\subsection{Ordinals} \label{ordinals}
The first example we consider is more easily computed in $\StNaamen$. The computations performed in this sub-section can be readily adapted to $\QtNaamen$ (with minor modifications), however, we were not able to find a natural set theoretic interpretation of these computations in the setting of $\Qt$. 

Consider $Sets^-$, the full sub-category of $\StNaamen$, whose objects are precisely those objects of $\StNaamen$ which happen to be sets. Consider the class function $S\mapsto S\cup \{S\}$ defined on $Sets^-$ (in fact, restricted to the category $Sets^-$ this is a functor). Indeed, an object of $\StNaamen$ is a set precisely if the operation $S\mapsto S\cup \{S\}$ is defined (in which case, of course, $S\lra S\cup\{S\}$ is a morphism in $\StNaamen$, and therefore also in $Sets^-$). Let us label those arrows by $(s)$. 

Observe that the function $S\mapsto S\cup \{S\}$ on $Sets^-$ allows us to define \emph{transitive} sets, namely:  a set $S$ is transitive precisely when $S\lra \{S\}$, or equivalently, if $S\cup \{S\}\lra S$, i.e., when the arrow $S\xrightarrow{(s)} S\cup \{S\}$ is invertible. Indeed, $S\lra \{S\}$ if and only if $s\subseteq S$ for all $s\in S$, if and only if $S$ is transitive. Thus, $S\in \Ob\StNaamen$ is an ordinal if and only if $S\xrightarrow{(s)} S\cup \{S\}\lra S\lra On$, where $On$ is the class of ordinals. We do not know whether the class $On$, as an object of $\StNaamen$ is definable (in some reasonable sense) in $\StNaamen$, even when augmented by the $(s)$-labelling. We point out however, that at least on the face of it, since the membership relation is not recoverable in $\StNaamen$, isolating the object $On$ in $\Qt$ allows us only to identify those objects of $\Qt$ all of whose members are ordinals, but not necessarily ordinals themselves. 

Note also that our $(s)$-labelling allows us only to identify arrows $S\lra S\cup \{S\}$. Given such an arrow, the object $\{S\}$ can be recovered as the complement of $S$ in $S\cup \{S\}$, i.e. it is the unique object whose direct limit (in $\StNaamen$)  with $\{S\}$ is $S\cup \{S\}$ and whose inverse limit with $S$ is $\0$. Of course, the arrow $\{S\}\lra S$ never exists, as it would imply that $S\subseteq s$ for some $s\in S$, so that $s\in s$ contradicting the regularity axiom of ZFC. 

In addition, by Proposition \ref{combinatorics}, if $S$ is an ordinal then $\0\rightwc S$ if and only if $S\le \aleph_0$ and $\0\rightc S$ if and only if $S\le \aleph_1$. So these two cardinals can also be recovered in $\StNaamen$ (with the function $S\mapsto S\cup \{S\}$), as the direct limits of the classes of (wc) and (c)-arrows respectively. This is, of course, not surprising, since $\StNaamen$ was constructed to model the notions of finiteness and countability. 

Of course, an ordinal $\alpha$ is limit precisely when $\beta\cup \{\beta\}\lra \beta \cup \{\beta\cup \{\beta\}\} \rtt \alpha \lra \top$ for all $\beta\in On$. It is a cardinal, precisely when for any ordinal $\beta$, if $\{\beta\}\rightc \{\alpha\}$ then $\{\alpha\}\rightc \{\beta\}$, which can be written as $\0\lra \{\alpha\}\rtt \{\beta\}\rightc \{\alpha\}$. Finally, $\alpha$ is a regular cardinal precisely when (it is a cardinal and) $\alpha\rightwf   \{\alpha\}$. To see this last claim, recall that, by construction, $\alpha\rightwf \{\alpha\}$ if and only if $\{A\}\rightc \{B\}\rtt \alpha\lra \{\alpha\}$ for all $\{A\}\rightc \{B\}$. But for sets, $\{A\}\rightc \{B\}$ if and only if $\card A+\aleph_0=\card B+\aleph_0$. So the lifting property defining the (wf)-arrows assures that for any $B\subseteq \alpha$, if some $A\subseteq B$ of the same (infinite) cardinality satisfies  $\{A\}\lra \alpha$ then $\{B\}\lra \alpha$. But $\{B\}\lra \alpha$ implies that there exists $\beta<\alpha$ such that $B\subseteq \beta$, so $B$ is bounded in $\alpha$. The other direction works in a similar way. 

As explained above, the operation $S\mapsto S\cup \{S\}$ is ``external'' to $\QtN$ (or $\StNaamen$). As a side remark to this subsection we point out that some traces of it can be recovered in a more ``geometric'' way within these labelled categories. Consider, for example, the property $A=\{\{a\}\}$ for some set $a$. It is easy to see that if $A$ is of this form then $\0\lra Z\lra A$ implies that either $\0\cong Z$ or $A\cong Z$. Conversely, if any decomposition $\0\lra Z\lra A$ is such that $\0\lra Z$ is an isomorphism or $Z\lra A$ is an isomorphism then $A=\{\{a\}\}$ for some set $A$. So we define, 
\begin{defn}
	An arrow $X\lra Y$ is \emph{indecomposable} if whenever $X\lra Z\lra Y$ either $X\lra Z$ is an isomorphism or $Y\lra Z$ is an isomorphism.
\end{defn}
With this definition the above observation can be stated as: $A$ is of the form $\{\{a\}\}$ for some set $a$ if and only if $\0\lra A$ is indecomposable. Note also that while we do not know whether indecomposability can be expressed as a lifting property it is obviously invariant under graph automorphisms of $\QtN$ (or $\StNaamen$), and can therefore be thought of as an intrinsic property of these (labelled) categories. 

It follows, for example, that with this in hand the property of  $A$ being \emph{isomorphic} to a singleton (i.e., $A\cong \{a\}$ for some set $a$) can be stated as: for all $X,Y$  if for all $\{\{a\}\}$, $\0\lra \{\{a\}\} \rtt X\lra Y$ then $X\lra Y\rtt A\lra \top$. It will suffice, of course, to show that this statement is equivalent to $A\cong\{\bigcup A\}$. By definition, $\0\lra \{\{a\}\}\rtt X\lra Y$ for all $\{\{a\}\}$ is equivalent to the statement that $\bigcup Y\subseteq \bigcup X$, in particular - for any set $A$ - $\0\lra \{\{a\}\} \rtt A\lra \{\bigcup A\}$ for all $a$.  Therefore, if for all $X,Y$ the assumption that $\0\lra \{\{a\}\} \rtt X\lra Y$ for all $a$ implies that $X\lra Y\rtt A\lra \top$ we can apply this with $A=X$ and $Y=\{\bigcup A\}$ to get $A\lra \{\bigcup A\}\lra A$, as required. The other direction is obvious, since the definition is invariant under changing $A$ with an isomorphic object, and therefore, we may assume that $A$ is a singleton.

\subsection{Cofinal and covering families.}
 A class $A$ is $\subseteq$-cofinal in $B$ if $A$ is a sub-class of $B$ and for all $b\in B$ there exists $a\in A$ such that $b\subseteq a$. In that case we also say that $A$ \emph{covers} $B$. By definition, this happens precisely when $B\lra A$. Since $A$ is a sub-class of $B$ we automatically get $A\lra B$, so that $A$ is cofinal in $B$ precisely when $A$ is isomorphic to $B$,  which happens if and only if $A\rightcwf B$. If $B$ is a set, the cofinality of $B$ is the minimal cardinality of a $\subseteq$-cofinal subset. In our notation, this can be expressed as: 
\[
\cof (B,\subseteq) = \min\{ \card B'\,:\, B'\rightcwf B\}. 
\]

For a class $B$ we have $\varnothing\rightc B$ if and only if every element of $B$ 
is at most countable, and  $\varnothing\rightwc B$ if and only if every element of $B$ 
is finite.  If $S$ is a set then $\0 \rightc B \rightwf \{S\}$ if and only if $B$ covers $[S]^{\le \aleph_0}$, i.e., if the set of countable subsets of $S$ is covered by $B$. In addition $\0\rightc [S]^{\le \aleph_0}\rightf \{S\}$ and $\0\rightwc [S]^{< \aleph_0}\rightwf \{S\}$. 
Combining all of the above, we get that for a cardinal $\kappa$: 
\[\cov(\kappa, \aleph_1,\aleph_1,2)=\min\{ \card B'\,:\, \0\rightc B'\rightwf 
\{\kappa\}\,\}=
\inf\{ \card B'\,:\, \0\rightc B' \leftarrow B'' \rightwf \{\kappa\}\,\}
\]
where $\cov(\kappa,\aleph_1,\aleph_1,2)$ is the minimal cardinality of a family of countable subsets of $\kappa$ covering $[\kappa]^{\le \aleph_0}$ (see Subsection \ref{covering} for more details). This shows that the covering number $\cov(\kappa,\aleph_1,\aleph_1,2)$ has a simple model categorical interpretation. The main goal of this paper is to show that, in fact, the right-most formula in the above equation is not only a simple translation of this set theoretic notion, but arises naturally from a model categorical study of $\Qt$. It is the \emph{co-fibrantly replaced left-derived functor of the cardinality function from $\Qt$ to $On^\top$}. This will be explained in detail in the next section.

\def\rightcwf{\xrightarrow{(wcf)}} 

\subsection{Some non-set theoretic concepts}
We conclude with a few simple non-set theoretic statements that can be expressed in $\Qt$. Consider, for example, $\CN$ a monster model of some first order theory $T$. For a cardinal $\beta$ let $\CN_{\beta}$ be the set of all elementary sub-models of $\CN$ of cardinality at most $\beta$. Then $\CN_\beta\rightwf \{\CN\}$ is the statement that the Lowehnheim-Skolem number of $T$ is at most $\beta$. Namely, it states that every subset of $N$ of cardinality at most $\beta$ is contained in a model of size at most $\beta$. In particular, if $T$ is countable then $\0\rightc \CN_{\aleph_0} \rightwf \{\CN\}$ is the statement that every countable set is contained in a countable model. Of course, the objects $\CN_\beta$ do not seem to be endemic to the model categorical setup.  


Recall that if $X$ is a topological space, then a set $A\subseteq X$ is closed if and only if $\acc(A)\subseteq A$, namely, if $A$ contains all its accumulation points. Thus,  a topological space can be given (instead of giving a collection of closed sets) by giving,  to any subset $S$ of $X$ the collection $\acc(S)$.  Therefore, a topological space $X$ gives rise to a functor $\acc:\QtNaamen\lra Sets^-$ by $\mathcal S \mapsto \{\acc(S\cap X): S\in \mathcal S\}$, and the topology on $X$ can be recovered from the functor $\acc$ only - namely, this gives a purely category theoretic definition of the topology.

\section{ Functors and derived functors.}

The idea of "forgetting structure" is, of course, a central theme in
mathematics. The pigeon-hole principal is about forgetting all the
structure but  cardinality, the dimension of algebraic extensions of
fields is defined by forgetting the field structure on the larger
field, keeping only the linear structure etc. In some sense, one could
argue that algebra and topology (when applied in the solution of
mathematical problems in other fields) act as powerful tools for
forgetting irrelevant information. In homotopy theory, it is common to
forget information that is not homotopy invariant. Thus, e.g., 
interesting homology theories are obtained by restricting scalars - the  
functor taking $R$-modules to $S$-modules for $S$ a sub-ring of $R$, 
or by considering the functor of global sections, which given a sheaf, forgets all its information but its global part. In model categories, 
Quillen's axiomatization of
homotopy theory, structure is forgotten by {\em deriving} functors, i.e. given a
functor $F$ the (left) derived functor $\LL F$ is the homotopy invariant
functor "closest" to $F$ (from the left). Here ``closest to $F$'' is interpreted
as being universal among the \emph{homotopy invariant} functors such that
there exists a natural transformation (also known as 
morphism of functors) $G \rra F$, i.e., if $\LL F$  is the derived functor of $F$ there exists a natural 
transformation $\LL F \rra F$ and any natural transformation $G \rra F$
from a homotopy invariant functor $G$, 
factors uniquely via $\LL F \rra F$. 

Homotopy invariance is defined
with the help of the \emph{homotopy category}. The homotopy category, $\rm{Ho}\FC$, associated with a model category $\FC$ is the category obtained from $\FC$ by formally inverting all weak equivalence so that
weak equivalences, and only weak equivalences,  become isomorphisms in $\rm{Ho}\FC$. A functor $F$ from a model category $\FC$ is \emph{homotopy invariant} if it factors trough $\rm{Ho}\FC$. 

Another means of forgetting structure in a model category $\FC$ is obtained by restricting to the sub-category of co-fibrantions (i.e., by forgetting all arrows that are not (c)-labelled. Note that since a model category is a labelled category, if we want - in addition - the resulting category to have an initial object, we have to restrict ourselves further to the category of co-fibrant objects (namely those objects $X$ such that $\0\rightc X$), and whose only morphisms are co-fibrations. It turns out that some constructions are well behaved only when restricted to 
this category of co-fibrant objects. 

Observe that by axiom (M2), given a model category $\FC$ any object $X$ is isomorphic in the homotopy category $\rm{Ho}\FC$ to a co-fibrant object, $X_{(wf)}$ such that $\0\rightc X_{(wf)} \rightwf X$. Thus, from the homotopy category point of view, every object can be replaced with a co-fibrant object, a process known as the co-fibrant replacement. For example, in the 
model structure on the chain complexes or sheaves this corresponds
to replacing a module or a sheaf by its projective (resp. injective)
resolution before computing cohomology.

\subsection{Functors and derived functors in quasi partial orders.}

As a category, $\QtNaamen$ is extremely simple: it is \emph{posetal}, namely, arrows are unique whenever they exist (so, e,g,, all diagrams commute). We call such categories posetal, for an obvious reason:  the relation $X\le Y$ defined by $X\lra Y$ is a partial quasi order encoding 
fully the category structure. In this sub-section we explain the details of the discussion of the previous paragraphs in the special case of posetal categories. 

Recall that given two categories $\FC$ and $\FC'$ a \emph{covariant functor} (or simply a functor) is a mapping $F:=(F_1,F_2):(\Ob \FC,\Mor \FC)\lra (\Ob \FC',\Mor\FC')$ sending the arrow $X\xrightarrow{f} Y$ to the arrow $F_1(X) \xrightarrow{F_2(f)} F_1(Y)$ and respecting commutative diagrams (namely, $F_2$ respects the composition of arrows). If both $\FC$ and $\FC'$ are posetal a functor $F:\FC \lra \FC'$ is merely an order preserving mapping.

If $F,G$ are such functors between (arbitrary) categories, $\FC,\FC'$, let $\epsilon(X)\in \Mor\FC'$ be an arrow $F(X)\lra G(X)$ (if such a morphism exists).  The collection of arrows  $\{\epsilon(X) : X\in \Ob\FC\}\subseteq \Mor(\FC')\}$ is a \emph{natural transformation} from $F$ to $G$ if $\epsilon(X)$ is defined for all $X\in \Ob\FC$ and 
it preserves commutative diagrams, i.e. all  diagrams involving $F,G$ and $\epsilon$, that 
exits purely for formal reasons, are necessarily commutative; in this case it means only that
$G(f)\circ \epsilon(X)=\epsilon(Y)\circ F(f)$ for every $X\xrightarrow{f} Y\in \Mor\FC$: 

\begin{figure}[H]
\centerline{
 \xymatrix  @R=3pc @C=3pc{
& X \ar[r]^f \ar[dl] \ar[d]  & Y \ar[d] \ar[dr] \\
F(X) \ar@/_1pc/[rr]_{F(f)} \ar[r]^{\epsilon(X)} & G(X) \ar@/_1pc/[rr]_{G(f)} & F(Y) \ar[r]^{\epsilon(Y)} & G(Y)
}}
\caption{There are two possible paths from $F(X)$ to $G(Y)$. If $\epsilon$ is a natural transformation then 
the composition of morphisms along those two paths are the same.}
\end{figure}

For posetal categories $\FC, \FC'$ and functors $F,G: \FC\lra \FC'$ the arrow $\epsilon(X)$ exists if and only if $F(X)\le G(X)$. Since the collection of functors from $\FC$ to $\FC'$ is itself quasi-partially ordered by pointwise 
domination, namely, for $F,G:\FC\lra \FC'$ write
\[
 F\le_N G \iff F(X)\le_{\FC'} G(X) \text{ for all } X\in \Ob\FC. 
\]
Thus, given two functors $F,G:\FC\lra \FC'$ there is a natural transformation from $F$ to $G$ precisely when $F\le_N G$. Moreover, if such a natural transformation exists, it is unique. 
Thus, for posetal $\FC,\FC'$ the functors  $F,G:\FC\lra \FC'$ are \emph{naturally equivalent} precisely when $F\le_N G \le_N F$, which - by definition - happens precisely when $F(X)$ is isomorphic (in $\FC'$) to $G(X)$ for all $X\in \Ob\FC$. 

As mentioned in the introductory paragraphs of this section, homotopy theory is interested in data only up to ``homotopy equivalence''. In the context of  model categories, $\FC,\FC'$, this means that those functors $F:\FC\lra \FC'$ a homotopy theorist is interested in are the ones respecting homotopy equivalence, i.e., those mapping (w)-labelled arrows to (w)-labelled arrows. Given a functor $F:\FC\lra \FC'$ which is not necessarily homotopy invariant, we may want to replace $F$ with a homotopy invariant relative, and naturally, we want this relative to be as close to $F$ as possible. We will now explain how this is done when $\FC'$ carries a trivial model structure. 

Recall that any category $\FC'$ can be given a  model structure by labelling all arrows (cf) and identifying the label (w) with isomorphisms. Call such a model structure trivial, and observe that any functor between model categories carrying a trivial model structure is homotopy invariant. Now assume that $\FC, \FC'$ are model categories with $\FC'$ trivial. Then $F:\FC\lra \FC'$ is homotopy invariant precisely when it maps weak equivalences to isomorphisms. Since the homotopy category $\rm{Ho}\FC$ is the localization of $\FC$ at the class of weak equivalences, this means precisely that $F$ is homotopy invariant if an only if $F$ factors through $\gamma$, where $\gamma:\FC\lra \rm{Ho}\FC$ is the localization functor. 

Thus, under the same assumptions, given any functor $F:\FC\lra \FC'$, the homotopy invariant version of $F$  we are looking for can be identified with a a functor $\L^\gamma_F: \rm{Ho}\FC\lra \FC'$ such that the composition $\L_F^\gamma\circ \gamma$ is ``closest'' to $F$. Formally, this last condition is interpreted as the existence of a natural transformation from $\L^\gamma_F\circ \gamma$ to $F$, and such that if $G:\rm{Ho}\FC\to \FC'$ is any functor such that there is a natural transformation from $G\circ \gamma$ to $F$ then there is also a natural transformation from $G$ to $\L^\gamma_F\circ \Gamma$.
In the case that $\FC, \FC'$ are posetal this reduces to: 
\begin{definition}\label{n:inf-functor}
 Let $\FC$ be a posetal model category, and $\FC'$ any posetal category. Given a functor $F:\FC\to \FC'$ the left derived functor of $F$ is given by
\[
 \L^\gamma F(X') =\inf\{ F(X) : X' \leq_{\rm{Ho}\FC} \gamma (X),\ X\in \Ob \FC\}. 
\]
In particular, the left derived functor exists if and only if the right hand side is well-defined.
\end{definition}

Observe that the definition of a derived functor $F:\FC\lra \FC'$ depends on $\FC$ being a model category only in as much as the homotopy category $\rm{Ho}\FC$ is the category through which we want to factor (an approximation) of $F$. In general, given a category $\mathfrak D$ and a functor $\gamma:\FC\lra \mathfrak D$,  we can still (left) derive any functor $F:\FC\lra \FC'$ with respect to the functor $\gamma$. If all categories involved are posetal, the formula in the above definition still gives the left-derived functor with respect to $\gamma$. 

\begin{example}[co-limits of commutative diagrams as derived functors.]
 A commutative diagram $D:A\lra \mathfrak B$ is a functor
where $A$ is a (usually finite) partial order. Let $\gamma:A\lra \{\bullet\}$ be the functor 
sending $A$ to the category $\mathfrak P=\{\bullet\}$ consisting of a single object 
and a single morphism. By definition $\L^\gamma D$ 
is an object of $\mathfrak B$. For simplicity, we abuse notation and denote this object $\L^\gamma_D$.  We claim that $\L^\gamma_D$ (if it exists) is the colimit of the diagram $D$. Indeed, given a
a functor $G:\{\bullet\} \lra \mathfrak B$, a natural transformation from $G\circ \gamma$ to $D$ is a collection of arrows $\{\epsilon(a):a\in\Ob A\}$ such that $G(\bullet)\xrightarrow{\epsilon(a)} D(a)$ and such that if $a\xrightarrow{f} b$ (some $a,b\in \Ob A$) then $\epsilon(b)=D(f)\circ \epsilon(a)$. The universality property of $\L^\gamma_D$ implies that given such a functor $G$ and a natural transformation as above, there exists a unique natural transformation from $G$ to $\L^\gamma_D$, i.e., there exists a unique arrow from $G(\bullet)$ to $\L^\gamma_D(\bullet)$, making the whole diagram commute. The latter is the defining property 
of the (direct) co-limit of $D$. 

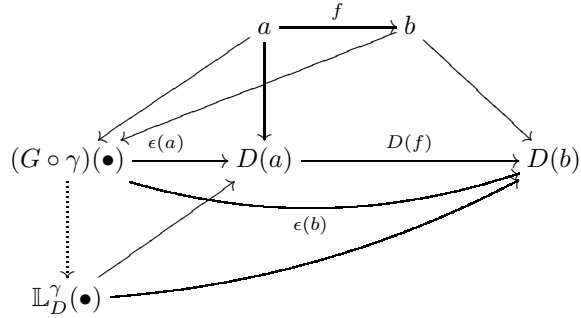
\begin{figure}[H]

\centerline{
 \xymatrix  @R=3pc @C=3pc{
& a \ar[r]^f \ar[dl] \ar[d]  & b \ar[dll] \ar[dr] \\
(G\circ \gamma) (\bullet) \ar@{.>}[d] \ar@/_1.5pc/[rrr]_{\epsilon(b)} \ar[r]^{\epsilon(a)} & D(a) \ar[rr]^{D(f)} &  & D(b) \\
\L^\gamma_D(\bullet) \ar[ur] \ar@/_1pc/[urrr] 
}}
\caption{$\L^\gamma_D$ is the direct colimit of the diagram $D$}
\end{figure}

Similarly, the co-limit of a commutative diagram $D$ is the right derived functor of the diagram 
viewed as a functor. 

Note that the localization functor $\gamma:A\lra \{\bullet\}$, up to
equivalence, corresponds to the degenerate model structure on $A$ where every
arrow is labelled $(wc)$ and only isomorphisms are labelled $(f)$ (or, alternatively, every arrow is labelled $(wf)$ and only isomorphisms are labelled
$(c)$). For this model structure every arrow in the homotopy category of $A$ is 
an isomorphism and this category is equivalent to $\{\bullet\}$.
\end{example}

\begin{example}[A category theoretic view of ordered sets\footnote{The authors thank Marco Porta 
for these observations.} ]
We have seen that (quasi) partially ordered sets can be viewed as a category. In this way any subset of such a partial order determines uniquely a commutative diagram. The colimit of such a diagram is the least upper bound of its vertices, and its limit is the greatest lower bound - if they exist.  

Say that the limit or colimit of  a diagram $D$ is {\em degenerate} if it is isomorphic to one of its vertices. It is then easy to check that a partially ordered set is {\em linear} if and only if either all finite limits or all finite colimits (exist and) are degenerate, (and, obviously, one implies the other, i.e. the limit of every finite diagram is degenerate if and only if the colimit of each degenerate diagram  is degenerate). A linearly  ordered set 
is {\em complete} if and only if every diagram has a limit if and only if every diagram has a colimit. A partially ordered class 
is {\em well-ordered} if and only if limits always exist and are always degenerate.  

We conclude with  a category theoretic characterisation of degenerate limits. 
 Obviously, any functor preserves all degenerate limits and colimits. We will show that, conversely, a limit preserved under all functors is degenerate. More precisely, if $A$ is a quasi-partial order, 
$D$ a commutative diagram, if for any functor $F:A\lra A'$ it holds that 
$\lim F(D)$ and $F(\lim D)$ are isomorphic (and in particular the former exists if and only the latter does),
then the limit of $D$ is degenerate. Indeed, To see this, consider the category $A'$, $\Ob A'=\Ob A \cup L$, and for $X,Y\neq L$ $\Mor_{A'}(X,Y)=\Mor_{A}(X,Y)$,
$\Mor_{A'}(X,L)=\Mor_A(X,\lim D)$. Finally, set $\Mor_{A'}(L,X)\neq\varnothing$ if and only if 
for some vertex $D_i$ in $D$ it holds $\Mor_A(D_i,X)$. We leave it is as a trivial exercise to check that $A'$ is indeed a (posetal) category and that $L$  and $\lim D\in \Ob A'$ are not isomorphic unless the limit of $D$ is degenerate.
 
 In particular, a quasi partially ordered class $A$ is well-ordered if and only if for every diagram $D$ and 
every functor $F:A \lra B$, it holds $F(\lim D) =\lim F(D)$, i.e. $\lim D$ exists if and only if $\lim F(D)$ 
exists, and if they both exists, the formula holds.
\end{example}

For what follows it is crucial to observe that the formula of Definition \ref{n:inf-functor} is meaningful whenever we are given a function $F:\FC\lra \FC'$ - not necessarily a functor. Thus, for example. if $\FC'$ is well ordered or if $\FC'$ is Dedekind complete then any function $F:\FC\lra \FC'$ can be derived from the left. 

Let $On^\top$ be the posetal category of ordinals (i.e., given ordinals $\alpha,\beta$  there is an arrow $\alpha\lra \beta$ if and only if $\alpha\le \beta$) augmented by a terminal object $\top$. The following definition sums up the discussion of the previous paragraphs, with an extra edge: 

\begin{definition}\label{crldf}
  Let $\FC$ be a posetal model category. For a {\sl function} $F: \FC \lra On^\top$ we let the \emph{co-fibrantly replaced left derived functor} of $F$ be: 
\[
\Lc F (X) = \min \left\{ F(Y): \vcenter { \xymatrix{
& X_1 & & X_3 & & X_n \ar@{-->}[r] & Y \\
X \ar@{-->}[ur] & & \ar@{-->}[ul]_{(w)} X_2 \ar@{-->}[ur] & & \ar@{-->}[ul]_{(w)} \cdots \ar@{-->}[ur] & & \;\perp \ar@{-->}[u]_{(c)}
} } \right\} 
\]
where the minimum is taken over all finite sequences of the same form. 
\end{definition}

Observe that given $X,Y\in \Ob\FC$ a sequence of the form 
\begin{equation*}\tag{$\diamond$}
X\lra X_1 \leftw X_2 \lra X_3\leftw \dots \lra X_n \lra Y 
\end{equation*}
  exists for some $n\in \mathbb N$ if an only if there exists $g\in \Mor\rm{Ho}\FC$ such that $X\xrightarrow{g} Y$. Thus, we can write: 
\begin{equation*}\tag{$\clubsuit$}
 \Lc F(X)=\min\{F(Y): X\lra_{h} Y, \perp\rightc Y\}, 
\end{equation*}
where $X\lra_h Y$ means that that there is an arrow from $X$ to $Y$ in the homotopy category. Using this notation we immediately see that $\Lc F$ is a homotopy invariant (because it factors through the homotopy category) \emph{functor} (because $X\lra Y$ implies that $\Lc F(X)\ge \Lc F(Y)$) depending only on the values $F$ takes on co-fibrant objects. 

Note that if $F:\FC\lra On^\top$ is a functor then for any $X\in \Ob\FC$, letting $\perp \rightc X_{(wf)} \rightwf X$ we see that $F(X_{(wf)})\le F(X)$, but $\Lc F$ is functorial so $\Lc F(X_{(wf)}) \le \Lc F(X)$. By what we have just said $X_{(wf)}\lra X$ implies that $\Lc F(X_{(wf)})\ge \Lc F(X)$, so - in the case $F$ is a functor: 
\[
 \Lc F(X)=\min\{F(Y): X\lra_h Y\}=\L^\gamma\circ \gamma(X).
\]
Thus, the co-fibrantly replaced left-derived functor generalises the definition of (left) derived functors (but the two definitions need not agree if $F$ is not a functor). 

\begin{rem}\label{coherence}
Let $\FC$, $\FC'$ be equivalent model categories, witnessed by the functors $F:\FC\to \FC'$ and $G:\FC'\to \FC$. Assume that $f:\FC\to On^{\top}$ is any function, then there is no reason to expect that $\Lc f(G(Y)))=\Lc (f\circ G)(Y)$. This is, of course, not the case if $f$ is a functor.  In other words, the price for deriving  arbitrary functions is that the process is not invariant under equivalence of model categories. This is discussed further in Section \ref{further}. 
\end{rem}

We will discuss the co-fibrant replacement a little more later on. Our discussion up to this point should convince the reader that - at least for functors - the left derived co-fibrant replacement is a natural means for forgetting non-homotopy-invariant information.

\subsection{Examples of derived functors.}

\subsection{The covering number of $\aleph_\omega$ as a value of a derived functor.} \label{covering}
In this sub-section we prove the main result of this paper, we show that the covering number of $\aleph_\omega$ is the value of the co-fibrantly replaced left-derived functor of the cardinality function $\card: \QtNaamen \lra On^\top$. 
Cardinality is certainly one of the most natural functions anyone studying set theory is bound to run into. Possibly, it  is the simplest set theoretic function not arising directly from purely logical operations (in the way the union and intersection operations do). To adapt the notion of cardinality to our setting we define a function $\card: \QtNaamen \longrightarrow On^\top$ such that  $X \longmapsto\card(X)$ if $X$ is a set and $X\longmapsto \top$ otherwise. 
Observe that cardinality is not a functor on $\QtNaamen$. Indeed 
$ \{X\}  \longrightarrow \mathcal P(X) \longrightarrow  \{X\} $ but $\card (\{X\})=1 < \card (\mathcal P(X)) > 1$ for all non-empty $X$. Similarly,

\[
\{\{\bullet_1\},
\{\bullet_1,\bullet_2\}\} \xrightarrow{(wcf)}  \{\{\bullet_1,\bullet_2\}\}
\]
is an isomorphism but 
$2=\card \{\{\bullet_1\},\{\bullet_1,\bullet_2\}\} > \card \{\{\bullet_1,\bullet_2\}\}=1$ 
are non-isomorphic.

However, cardinality is a natural function and the homotopy ideology discussed in the introduction to this section suggests (despite of the fact it is not a functor) that we try and find a homotopy invariant
approximation to cardinality. As discussed above, any function from a model category to $On^\top$ can be derived. Unfortunately, as we will see later, deriving the cardinality function (according to the formula in Definition \ref{n:inf-functor}) gives us an uninteresting result. So we take the co-fibrantly replaced left derived functor of cardinality, as in Definition \ref{crldf}. The resulting function,  $\Lc\card$, can be viewed, as homotopy theory yoga suggests, as the homotopy invariant version of cardinality. 

Interestingly, the homotopy invariant version of cardinality has a purely set theoretic interpretation ($\Lc\card(\{\aleph_\alpha\})=\cov(\aleph_\alpha,
\aleph_1,\aleph_1,2)$  - where $\cov(\aleph_\alpha,
\aleph_1,\aleph_1,2)$ is the \emph{covering number} to be discussed in detail below). The construction of this function uses fairly little set theory: the only notions needed in an essential way to construct it are $A\subseteq B$, finiteness, countability and infinite equi-cardinality. Thus, $\Lc\card$ will remain meaningful in any set theory where those notions keep their meaning.  More importantly, $\Lc \card$ is considerably tamer, say, than the power function, and can be effectively bounded in ZFC (but these are deep results in PCF theory, and we do not claim that they can be identified, let alone proved - using homotopy theoretic tools). For example, Shelah's famous inequality 
\[
(\aleph_\alpha)^{\aleph_0}\leq \cov(\aleph_\alpha,\aleph_1,\aleph_1,2) + 2^{\aleph_0} 
\]
can be interpreted (paraphrasing Shelah) as a decomposition of $(\aleph_\alpha)^{\aleph_0}$ into a ``noise'' component (wild and highly independent on ZFC) and a ``homotopy invariant'' part, which can be well understood within ZFC. 

Another curious feature of the function $\Lc\card$ is that it is non-trivial only on singular cardinals. Thus, from a homotopy theoretic view point singular cardinals present themselves almost immediately as a natural object of interest in set theory (compare with, \cite{KojmanHistory}, describing the early and spectacular appearance of singular cardinals on the mathematical stage, and their immediate disappearance for several decades). We now proceed with a detailed exposition of the discussion of the last paragraphs. 

By definition, the \emph{covering number}

\[
\cov(\lambda,\kappa,\theta,\sigma)
\]
is the {\em least size of a
family $X\subseteq [\lambda]^{<\kappa}$
of subsets of $\lambda$ of cardinality less than $\kappa$,
such that every
subset of  $\lambda$ of cardinality less than $\theta$,
lies in a union of less than $\sigma$ subsets in $X$.}

\begin{thm} \label{7.2}(the covering number as a derived functor). 
{\em For any cardinal $\lambda$ 
\[
\L_c\card (\{\lambda\}) =cov(\lambda,\aleph_1,\aleph_1,2) 
\]
}
\end{thm}

\begin{proof}
 First, assume that $\CY$ is a covering family for $\lambda$ witnessing $\cov(\lambda,\aleph_1,\aleph_1,2)=\kappa$. Then, by definition of the covering number $\0\rightc\CY$. We claim that $\CY\rightw \{\lambda\}$, which will prove $\Lc\card({\lambda})\le \kappa$. By Proposition \ref{combinatorics} we only have to show that any countable subset of $\lambda$ is contained in an element of $\CY$, which is merely the definition of $\CY$ being a covering family. To prove the other inequality, observe that: \\

\noindent{\bf Claim I}:  If $\CX\lra_h\CY$ in $\QtNaamen$, with $\CX:=\CX_0,\CX_1,\dots,\CX_n=:\CY$ witnessing it (as in  ($\diamond$)) then for every $i\le n$, every countable subset $L$ with  $\{L\}\lra \CX$ is contained, up to finitely many elements, in some $\{X\}\lra \CX_i$. 

\proof
For $X_0$ there is nothing to prove, and for $X_1$ this follows from the definition of $\CX\lra \CX_1$. For $\CX_2\rightw \CX_1$ this is a special case of Proposition \ref{combinatorics}, and as the condition is transitive, induction gives this observation. \qed$_{\text{Claim I}}$ \\

\noindent The proof of the theorem now follows from the following: \\

\noindent{\bf Claim II}: Let $\0\rightc \CY$ be such that $\{L\}\lra^* \CY$ for every countable set, $L\subseteq \lambda$. Then there exists a covering family $\CZ$ of $\lambda$ whose cardinality is at most that of $\CY$. \\ 

\proof Let $\CY_0$ be the inverse limit of $\CY$ and $\{\lambda\}$. Then $\card \CY_0\le \card \CY$ (by definition of the inverse limit in $\Qt$), and   $\0\rightc \CY_0$. By assumption,  $\0\rightc \CY_0\rightwc \lambda_{(wf)}$, where $\0\rightc \lambda_{(wf)}\rightwf \{\lambda\}$. By Lemma 35 of \cite{GaHa} there is a set $\Lambda$ and some $\CY'$ such that the following diagram is true in $\StNaamen$: 
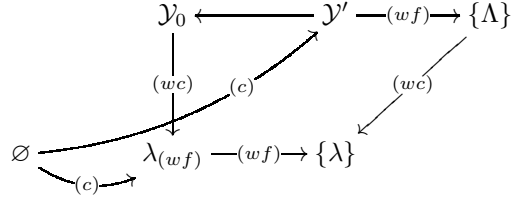
\begin{figure}[H]
\centerline{
 \xymatrix  @R=3pc @C=3pc{
& \CY_0 \ar[d]|-{(wc)} & \CY' \ar[l] \ar[r]|-{(wf)} & \{\Lambda\} \ar[dl]|-{(wc)}  \\
\0 \ar@/_1pc/[r]|-{(c)} \ar@/_1pc/[urr]|-(0.7){(c)} &  \lambda_{(wf)} \ar[r]|-{(wf)} & \{\lambda\}  
}}
\end{figure}
Since $\QtNaamen$ is a full sub-category, to show that this diagram is also true in $\QtNaamen$ it suffices to verify that all objects in the diagram are objects in $\QtNaamen$.  This amounts to checking that $\{\lambda\}$ and $\CY'$ are in $\QtNaamen$, which is obvious since all singleton sets and all co-fibrant objects are. 

Let $\CY''$ be the inverse limit of $\CY_0$ and $\{\Lambda\}$, by Claim 33 of \cite{GaHa} (and this also follows readily from the definition), this is simply $\{y\cap \Lambda: y\in\CY_0 \}$. By definition of the inverse limit we get $\CY'\lra \CY''$. Since $\CY'\rightwf \{\Lambda\}$ it follows (e.g., by Proposition \ref{combinatorics}) that $\CY''\rightwf \{\Lambda\}$. Since all elements in $\CY'$ are countable so are all the elements in $\CY''$. By Proposition \ref{combinatorics} these two facts together mean precisely that $\CY''$ is a covering family for $\Lambda$. 

Finally, since $\{\Lambda\}\rightwc \{\lambda\}$ we get (again, using Proposition \ref{combinatorics}), that $\lambda\setminus \Lambda$ is a finite set, say, $C$. Let $\CZ:=\{y\cup C: y\in \CY''\}$. Then $\CZ$ is a co-fibrant object, and is therefore an object of $\QtNaamen$. All elements in $\CZ$ are countable, and every countable subset of $\lambda$ is contained in an element of $\CZ$. So $\CZ$ is a covering family for $\lambda$. Observe that $\card\CZ\le \card \CY''\le \card \CY_0=\card\CY$. Thus, $\CZ$ witnesses that $\Lc\card\ge \cov(\lambda,\aleph_1,\aleph_1,2)$. 
\qed$_{\text{Caim II}}$\\

\noindent This completes the proof of the theorem. 
\end{proof}

We conclude with a summary, in our notation, of some of Shelah's results concerning PCF bounds: 

\begin{thm}[Shelah]
The following inequalities are true in ZFC: 
\bi
\item[(i)]if  $\aa$ is regular cardinal, then
 \[
  \L_c(\{\aa\})=\L_c(2^{\aa})=\cov(\aa, \aleph_1,\aleph_1,2)=\aa
 \]
\item[(ii)]
 $\L_c(\{\aleph_\omega\})=\L_c(2^{\aleph_\omega})=\cov(\aleph_\omega,\aleph_1,\aleph_1,2)<\aleph_{\omega_4}
 $ \item[(iii)] If $\aleph_\delta$  is a singular cardinal such that
 $\delta<\aleph_\delta$, then 
$\L_c(\{\aleph_\delta\})=\L_c(2^{\aleph_\delta})=\cov(\aleph_\delta,\card
\delta^+,\card \delta^+,2)<\aleph_{\card\delta^{+4}}$ \item[(iv)] (Shelah's
Revised GCH). If $\theta$  is a strong limit uncountable cardinal, then for
every $\lambda \geq \theta$, 
$\kappa_0 \leq\kappa<\theta$ 
$\lambda^{[\kappa]}=\lambda$

\ei \end{thm}\parno*{{\bf Proof.}} (i) is immediate by induction; (ii) is a
particular case of $(iii)$; (iii) is Theorem 7.2 of [Handbook of Set theory,
page  1209] ; (iv) Theorem 8.1 of [Handbook of Set theory,  page 1210]

Note that we do not say anything about the {\em fixed points}
$\alpha=\aleph_\alpha$ of $\aleph_{\bullet}$-function. (todo: is there an
explanation)

\subsection{Other model categories and covering numbers}\label{other}

Simple variations on the theme leading us to ``rediscover'' the covering
number $\cov(\lambda,\aleph_1,\aleph_1,2)$ result in other covering numbers.
Since most of the details are quite similar, we will be brief. 

For an object $A$ of $\QtNaamen$, let $\QtNaamen^{A}$ be the full sub-category
of arrows $A\lra X$ with the induced model structure, i.e., the full
subcategory of $\QtNaamen$ consisting of those objects $X$ such that $A\lra X$
with the labelling induced from $\QtN$. This is, trivially, a model category.
Applying Definition \ref{n:inf-functor} for $\QtNaamen^{A}$ and the function
$\card:\QtN^A\lra On^T$ to obtain the functor \[ \L_c^{Qt^{A}}\card :
\QtNaamen^{A} \lra On^\top, \] the cofibrantly replaced left-derived functor
of cardinality (on the model category $\QtNaamen^{A}$). We obtain: 
\begin{thm}
	Let $\beta\le\alpha$ be ordinals. Let
	$\ab*:=[\aleph_\beta]^{<\aleph_\beta}$. Then, with the above notation, if
	$\aB$ is regular then \[ \L_c^{Qt^{\ab*}}\card(\{\aa\})=\cov(\aa,\aB,\aB,2).
	\] In particular, if $\aa<\aleph_{\aa}$ and $\aB=(\cof \aa)^+$ then \[
	\L_c^{Qt^{\ab*}}\card(\{\aa\})=\pp_{\cof\aleph_\alpha}(\aleph_\alpha)=\pp(\aleph_\alpha)\]
\end{thm}
	
\[
 \{\aa\}\lra \{\aa\} \xleftarrow{(w)}\CY\times \{\aa\}\lra \CY
\xleftarrow{(c)} \pb.  
\]

Thus, $\L_c^{Qt^{\ab*}}\card(\{\aa\})\le \kappa$.  So we now turn to the proof
of the other inequality. Let $\{\aa\}\lra_h \CY$ for some co-fibrant object
$\CY$ of minimal cardinality. We first prove: \\

\noindent{\bf Claim I} Let $L\subseteq \aa$ be any set with $\card L<\aB$.
Then there exists some $L'\subseteq L$ such that $L\setminus L'$ is finite and
such that $\{L'\}\lra \CY$.  We denote this property
$[\aa]^{<\aB}\lra^*\CY$.

\begin{proof} By Claim I of the previous theorem we know that every countable subset
of $L$ is contained, up to a finite set, in some element of $\CY$. That is, if
$L_0\lra  L_c$ (where $\0\rightc L_c\rightwf \{L\}$) then $L_0\lra^* \CY$.
Letting $\CY_0$ be the inverse limit of $\CY$ and $L$ this means that
$\CY_0\rightwc L_c$. Therefore, as in the proof of Claim II of the previous
theorem, we may apply Claim 33 of \cite{GaHa} to obtained $L'$ satisfying the
requirements. \qed$_\text{Claim I}$ \\

To conclude the proof of the theorem we need one additional combinatorial
fact, a generalisation of Lemma 35 of \cite{GaHa}: \\

\noindent{\bf Claim II} Assume that $[\aa]^{<\aB}\lra^*\CY$.  Then there is a
finite set $B$ such that $\aa^{<\aB}\lra \CY_B$, where $\CY_B$ is the set
$\{Y\cup B: Y\in \CY\}$. 

\proof Assume not. Then for any finite $b\subseteq \aB$ there exist $L_b\in
[\aa]^{<\aB}$ and $L\not \lra \CY_b$. Let $L_0=L_{\0}$. Define by induction:
\[ L_{i+1} = L_i\cup \{L_b: b\in [L_i]^{<\aleph_0}, L_b\not \lra \CY_b\}.  \]
Let $L=L_\omega:=\bigcup_{i<\omega} L_i$. Because $\aB$ is regular, $\card
L<\aB$. But there is no finite set $b\subseteq L$ such that $L\lra \CY_b$.
Indeed, if $b$ were such a set, then $b\subseteq L_i$ for some $i<\omega$. So
$\{L_b\}\lra \{L_{i+1}\}$ and $\{L_b\}\not \lra \CY_b$. Since $\{L_{i+1}\}\lra
\{L\}$ it follows that $L\not \lra \CY_b$, a contradiction. \qed$_\text{Claim
II}$ \\

Let $B$ be a finite set as in Claim II, then $\CY_B$ covers $\aa^{\aB}$, and
$\card \CY_B=\card \CY$. Because $B$ is finite and $\CY\in \Ob\Qt$ it follows
immediately from the definition that $\CY_B\in \Ob\Qt$, with the desired
conclusion.  \end{proof}

The construction of the co-slice category, $\Qt^A$, for an object $A\in
\Ob\Qt$ is standard in category theory. We proceed now to a slightly different
construction, to our taste quite natural from the set theoretic point of view,
but not entirely obvious from on the category theoretic side: 

Let $X$ be a class of sets, fix a (regular) cardinal $\kappa$ and  denote
$\ukappa X:=\{\bigcup S \,:\, S\subseteq X, \card S < \kappa\}$. Call a class
$X$ of sets $\kappa$-directed if $\ukappa X\lra X$, namely  if any  collection
of less than $\kappa$ members of $X$ has a common upper bound (with respect to
$\subseteq)$) in $X$.  Let $\qtkappa$ be the full subcategory of $\StNaamen$
consisting of $\kappa$-directed classes.

Let $\qtkk$ be a category that has the same object as QtNaamen,
$\Ob\qtkk=\Ob\Qt$, and $X\lra Y$ in $\qtkk$ if and only if $\ukappa
X\lra\ukappa Y$.  Given $X\in \Ob\qtkk$  denote $F(X):=\ukappa X$. It is clear
that $F:\qtkk\to \qtk$ is a functor. Moreover the inclusion mapping $G:\qtk\to
\qtkk$ given by $G(X)=X$ is a functor (as for any $X,Y \in \Ob\qtk$ if $X\lra
Y$ then $\ukappa X\lra \ukappa Y$). By definition, for $X\in \Ob\StNaamen$,
$X\longleftrightarrow \ukappa X$, so the functors $F$ and $G$ show that $\qtk$
is equivalent to $\qtkk$.

It is easy to check that for regular $\kappa$ the category $\qtkappa$ equipped
with the following labelling satisfies Quillen's axioms (M1)-(M4) and (M6): 

\begin{defn}\label{qtk}

\bi  

\item $X\lra Y$ iff $\forall x\in X\exists y\in Y \, x\subseteq y$ 

\item $X\rightwc Y$ iff $\forall y\in Y\exists x\in X \,(\card( y \setminus x)
	<\kappa )$ (and $X\lra Y$)

\item $X\rightc Y$ iff $\forall x\in X\exists y\in Y  ( \card y\leq \card x +
	\kappa )$ (and $X\lra Y$)

\item $X\rightf Y$ iff $\forall x\in X \forall y'\subseteq y \in Y\exists
	x'\in X (\card y'< \kappa\implies x\cup y'\subseteq x')$ (and $X\lra Y$)

\item $X\rightwf Y$ iff $\forall x\in X \forall y'\subseteq y\in Y \exists
	x'\in X (\card y'\leq \card x+\kappa\implies y'\subseteq x')$ (and $X\lra
	Y$)

\item $X\rightw Y$ iff $\forall x\in X \forall y'\subseteq y\in Y \exists
	x'\in X 
	\card( y'\setminus x')<\kappa)$ (and $X\lra Y$)

\ei

\end{defn}

\begin{rem}\label{intersect} Observe that $X\rightwc Y$ ($X\rightwf Y$) if and
	only if $X\rightc Y$ ($X\rightf Y$) and $X\rightw Y$. Moreover, $X\rightw Y$
	if and only if there exists $Z$ such that $X\rightwc Z\rightwf Y$.
\end{rem}

To turn $\qtk$ into a model category, as with $\StNaamen$, let $\Qtk$ be the
full sub-category of \emph{cute} objects of $\qtk$, namely, those objects
satisfying the diagram of Figure \ref{Qt} (with respect to the labelling in the
above definition). Now one defines, for $X\in \Ob\qtk$, $\tilde X$ to be the
product of all cute $Y\in \Ob \qtk$ such that $X\lra Y$ (we leave it as an
exercise to verify that this is indeed an object in $\StNaamen$). It is then
easy to verify that $\tilde X$ is cute and that if $\0\rightc X$ then
$X=\tilde X$ and that $\widetilde {\{S\}}=\{S\}$ for any set $S$. So $\Qtk$
satisfies Axiom (M0) (inverse limits are simply products, and the direct limit
$\{X_1,\dots,X_k\}$ is simply $\widetilde{\Sigma_{i=1}^k X_i}$, where $\Sigma
X_i$ is the limit of the $X_i$ in $\qtk$). That the remaining axioms are
satisfied in $\Qtk$ can be proved precisely as in \cite{GaHa}, with the
obvious adaptations (replacing "countable" there with "of cardinality at most
$\kappa$" and "finite" there with "of cardinality smaller than $\kappa$", and
see also Remark 37 in \cite{GaHa} for the fixed point argument needed for the
proof the analogue of Lemma 35).  

Recall that, as pointed out above, $\qtkk$ is equivalent to $\qtk$. This
equivalence can be used to label $\qtkk$ uniquely to make the two categories
equivalent as labelled categories. Since the definition of $\Qtk$ is given
strictly in terms of the labelling of $\qtk$, we obtain a full sub-category,
$\Qtkk$, of $\qtkk$, equivalent as a labelled category to $\Qtk$ ($\Qtkk$ is
both the image of $\Qtk$ under the functor mapping $\qtk$ into $\qtkk$ and the
full sub-category of cute objects of $\qtkk$ as a labelled category). Thus,
$\Qtkk$ is a model category equivalent to $\Qtk$. 

As $\Qtk$ is equivalent (as a model categories) to $\Qtkk$ so are their
associated homotopy categories. Computing the homotopy category of $\Qtk$ is
rather simple: objects are $<_{\kappa}$-directed classes with arrows $X\lra Y$
if and only if for all $x\in X$ there exists $y\in Y$ such that
$\card(x\setminus y)<\kappa$ (this follows immediately Definition \ref{qtk}
and the fact that $\rm{Ho}{\Qtkk}$ is obtained by inverting all $(w)$-arrows
in $\Qtkk$). 

It is now straightforward to verify that the left derived functor of $\card :
\Qtkk \lra On^\top$ is Shelah's revisited power function: \[
\Lc\card(\{\lambda\})=\lambda^{[\kappa]}:=\cov(\lambda,\kappa^+,\kappa^+,\kappa).\]
Indeed, $\cov(\lambda, \Delta, \theta, \sigma)$ is the least size of a family
$X \subseteq  [\lambda]^{<\Delta}$, such that every subset of $\lambda$ of
cardinality smaller than $\theta$, lies in a union of less than $\sigma$
subsets in $X$. In our notation, taking $\Delta=\kappa^+=\theta$ and
$\sigma=\kappa$, the condition on the family $X$ can be stated as: $ X \lra
[\lambda]^{\le \kappa}$ and $[\lambda]^{\le \kappa}\lra \bigcup_{<\kappa }X$.
Now, the first of these conditions is precisely $\0\rightc X\lra \{\lambda\}$,
whereas the second condition is $\bigcup_{\le \kappa} X \longleftarrow Y
\rightwf \{\lambda\}$ for some $Y$. But in $\qtkk$ (and therefore in $\Qtkk$),
$\bigcup_{<\kappa} X \longleftrightarrow X$. Therefore, this last condition is
equivalent to $X \longleftarrow Y \rightwf \{\lambda\}$. Combining everything
together we get that $\cov(\lambda, \kappa^+,\kappa^+,\kappa)\ge \Lc
\card(\{\lambda\})$. The proof of the other direction is similar (modulo the
obvious adaptations) to the proof of the analogous fact in Theorem \ref{7.2}.

The model category $\Qtkk$ allows us to formulate quite easily the notion of
the cardinal $\kappa$ being (non) measurable. Recall that a cardinal $\kappa$
is \emph{measurable} if it is uncountable and admits a $<k$-complete
non-principal ultrafilter, or, equivalently, 0-1 valued probability countably
additive measure such that every subset is measurable.  Such an ultrafilter
exists on he cardinal $\kappa=\omega$, as any filter is  $<\omega$-complete.
We prove that: 

\begin{lem}\label{measurable} The following are equivalent for a regular
	cardinal $\kappa>\omega$: \bi \item $\kappa$ is not measurable.  \item For
	all $X\in \Ob\Qtkk$ if $X\xrightarrow{(i)}\{\kappa\}$ then $X\rightw
	\{\kappa\}$.  \item $X\xrightarrow{(i)} Y \xleftarrow{(c)} \perp$ implies
	$X\rightwc Y$.  \item In  $\rm{Ho}\Qtkk$ if $X\xrightarrow{(i)} Y$ then
	$Y\lra X$ for all $X, Y$. 
	 \ei where $X\xrightarrow{(i)} Y$ means that $X\lra Y$ is an indecomposable
	 arrow.  \end{lem}
\begin{proof} First, we observe that if $X\in \Ob\qtkk$  and 
$\bar X=\bigcup_{<\kappa} X$ then
$\bar X\longleftrightarrow X$ (in $\qtkk$). In particular, $X\in \Ob\Qtkk$ if and
only if $\bar X$ is. Note that if $X\in \Ob\qtkk$ and $X$ is a non-empty set
then $\bar X$ is a $\kappa$-complete ideal on $\bigcup X$. Indeed, it is
closed under unions of size less than $\kappa$, and by definition $\bar X$ is
downward closed. 

Thus, if $X\in \Ob\qtkk$ then the statement $X\lra \{\kappa\}$ is equivalent
to $\bar X\lra \{\kappa\}$ 
and since $\{\bigcup X\}\in \Qtkk$ we get $X\lra \{\bigcup X\}\lra
\{\kappa\}$. If, in addition, $X\xrightarrow{(i)} \{\kappa\}$ then either
$\{\bigcup X\}\leftrightarrow X$ or $\bigcup X=\kappa$. But if $\bigcup X\neq
\kappa$ then $X\lra \{\kappa\}$ is not indecomposable (take $X\cup \{y\}$ for
any $y\in \{\kappa\}\setminus \bigcup X\}$). So $X\xrightarrow{(i)}\{\kappa\}$
is equivalent to $\bar X $ being a maximal ideal on
$\kappa$ which is also $\kappa$-complete. 

It remains, therefore, to ascertain when is such an ideal principal.   On the
one hand, it is obvious that if $X$ is a maximal principal ($\kappa$-complete)
ideal on $\lambda$ then $X\rightwc \{\kappa\}$. Now assume that $X$ is a
maximal ideal on $\{\kappa\}$ which is $\kappa$-complete. Then  $X\lra
\{\kappa\}$, which is, by Definition \ref{qtk}, equivalent to $X\rightc
\{\kappa\}$, and assume that $X\rightw \{\kappa\}$. Then $X\rightwc
\{\kappa\}$, which - by definition - means some $x\in X$ satisfies
$\card(\kappa \setminus x)<\kappa$, and since $X$ is $k$-complete this means
that $\bigcup X\neq \{\kappa\}$ (if it were, then already a small union would
cover everything). Maximality implies that in that case $X$ is principal.

The above shows the equivalence of $(1)$ and $(2)$ above, as well as
$(3)\Rightarrow (1)$. So it remains to prove $(2)\Rightarrow (3)$. Indeed,
assume that $X,Y$ are as in $(3)$. We may assume that $X=\bar X$. We may also
assume that there exists some $y\in Y$ such that $\card(y)=\kappa$ (otherwise
$X\rightwc Y$ is automatic from Definition \ref{qtk}). So fix any $y\in Y$. It
will suffice to show that $X_y:=\{x\in X: x\subseteq y\}$ is a maximal
$k$-complete ideal on $y$. This will be enough since then, by $(2)$ this ideal
is principal, and as $y\in Y$ was arbitrary of cardinality $\kappa$, we will
be done, by Definition \ref{qtk}. So it remains to show that $X_y$ is a
maximal ideal on $y$. That is it is a $k$-complete ideal is proved exactly as
above. So we only have to verify its maximality, which is immediate from the
indecomposability of the arrow $X\lra Y$.

To see the equivalence with (4) we may assume that $X$ and $Y$ are co-fibrant.
Further, note that $X\xrightarrow{(i)}_h Y$ if there exists a sequence as in
($\diamond$) in which all but one of the arrows is a $(w)$-arrow, and this
arrow is indecomposable. But, if \[ \perp \rightc X\rightw X_1 \leftw X_2
\rightw X_3\leftw \dots X_i \xrightarrow{(i)} X_{i+1} \rightw X_n \lra Y
\xleftarrow{(c)} \perp \] then in $\rm{Ho}\Qtk$ the object $X$ is isomorphic
to $X_i$ and $Y$ is isomorphic to $X_{i+1}$ and $X_{i+1}$ is co-fibrant. Thus,
(3)$\iff $(4).  \end{proof}

We point out that, since $\Qtkk$ is a closed model category (i.e., it
satisfies axiom (M6)),  condition (3) above can be expressed as a lifting
property: \[ \perp\rightc Y \Rightarrow X\xrightarrow{(i)} Y\rtt X'\rightf Y'.
\] Thus, the notion of $\kappa$ being a (non)-measurable cardinal has an
essentially model category-theoretic interpretation. 

\def\St{\StNaamen}

The following gives another intriguing set theoretic angle to the model
category $\Qtkk$. Let $L$ denote, as usual, G\"odel's constructible model of
set theory. Recall, e.g., Theorem \cite[Theorem 13.9]{}, that $L$ is the least
transitive class (i.e., $L\lra \{L\}$) closed under all G\"odel operations,
and universal in the sense that $X\subseteq L$ implies $X\subseteq Y\in L$.
Let $\St^L$ be the full sub-category of $\St$ whose objects are sub-classes of
$L$ definable within $L$. In other words, viewing $L$ as a model of ZFC and
ignoring the ambient universe $V$, we let $\Ob\St^L$  be all sub-classes of
$L$.  Note that as $L$ itself is a class  element in $\Ob\St^L$ is, in fact an
object of $\St$. 

\begin{rem}\label{nomeasurable} It is well known (and easy to prove) that if
	there is no measurable cardinal below $\kappa$ then $\kappa$ is measurable
	if and only if it admits a $\sigma$-complete ultrafilter, in other words, a
	countably additive measure such that every subset has measure either 0 or 1.
	Thus, the statement "there is no measurable cardinal" is equivalent to the
	statement "the only indecomposable arrows in $\St_{\aleph_1}$ are
	(wc)-arrows". 
\end{rem}

Obviously, by the remark concluding the previous paragraph, since $L\models
ZFC$,  we can perform the construction of \cite{GaHa} in $\St^L$. However,
since notions of countability and (infinite) equicardinality are not absolute,
the labelling obtained in this way will, in general, not coincide with the
labelling induced on $\St^L$ from $\St$.  Indeed, the labelling induced on
$\St^L$ from $\St$ will not (in general) satisfy axiom (M2) (whereas, the the
labelling following the construction of \cite{GaHa} does). 

In the above, it seems clear that the labelling induced on $\St^L$ from $\St$
is not the "right" one. The situation is less clear when trying to give
sub-categories of $\St^L$ a model structure. As mentioned above, carrying out
the construction of \cite{GaHa} we can construct $\Qt^L$ and the associated
model  categories $\Qtkk(L)$ (where $\kappa$ is a cardinal in $L$). But it
seems that under certain set-theoretic assumptions other model structures can
also be constructed. 

Let $\kappa$ be a (regular) cardinal in $V$. Let $\Qtkk(L)$ be the full
labelled sub-category of $\Qtkk$ whose objects are also in $\Ob\St^L$. It is
not hard 
to check that $\Qtkk(L)$ is closed under (small) limits, and that, being a
full sub-category of $\Qtkk$ it also satisfies ($M1$) and ($M3$)-($M6$)
(though ($M6$) requires a small calculation).  Recall that the
($M2$)-decomposition of each arrow in $\Qtkk$ is unique (up to
$\Qtkk$-isomorphism), so in order to check whether $\Qtkk(L)$ satisfies ($M2$)
we have to show that if either $X\rightc Z \rightwf Y$  or $X\rightwc Z
\rightf Y$ with $X,Y\in \Ob\Qtkk(L)$ then $Z$ is isomorphic to an element in
$\Ob\Qtkk(L)$. Now, recall that, by definition $Z$ is the class of all
$z\subseteq y\in Y$ such that $\card z\le \card x+\kappa$ for some $x\in X$,
or $Z$ is the class of all $z\subseteq y\in Y$ such that $\card (z\setminus
x)<\kappa$ for some $x\in X$. 
Observe that, since $X$ and $Y$ are definable within $L$ so is the class
$Z_L:=Z\cap L$ of all constructible members of $Z$. Thus, it will suffice to
show that $Z_L\longleftrightarrow Z$. Of course, in general, there is no
reason for this to be true. But if $\kappa>\aleph_1$ then this statement is
equivalent to the conclusion of Jensen's covering lemma (for $\kappa$). 

Thus, e.g., if $0^\#$ does not exist and $\kappa>\aleph_1$ then $\Qtkk(L)$ is
a model category. It is an easy exercise to check that there is no cardinal
$\lambda\in L$ such that $\Qtkk(L)$ is precisely the model category
$\Qt_{\lambda}^+$ constructed within $L$. We do not know whether $\Qtkk(L)$
could be equivalent to $\Qt_{\lambda}^+$ for some $\lambda$ (with the latter
constructed within $L$).

\section{Suggestions for future research}\label{further}

Among the possible objections to the work presented in the present paper and
its predecessor there are two which we view as most intriguing. These are the
coherence and usefulness of work. 

The problem which we call \emph{coherence} is that, as explained in Section
XYZ, if $f:\FC\lra On^{\top}$ is any function on the posetal model category
$\FC$, then the value of the (cofibrantly replaced) left-derived functor of
$f$ is not necessarily invariant under the equivalence of model categories.
Namely, if $\FC'\equiv \FC$ (as model categories) and $F:\FC'\to \FC$ is a
witness (of one direction of) this equivalence then $\Lc(F\circ f)(x)$ is not
necessarily the same as $\Lc f(F(x))$. 

This is most obvious in our calculation of Shelah's revised power function. In
deriving the cardinality function on $\Qtkk$ we obtain the desired result, but
if we tried doing the same on the equivalent model category  $\Qtk$, we would
have obtained a different answer. The same situation would have occurred if
trying to derive cardinality in $\QtNaamen$ we worked with the (equivalent)
full subcategory of "downward closed" objects. 

Because the derivation of a function $f:\FC\lra On^{\top}$ on a posetal model
category can be viewed as a minimization operation (of $f(x)$ over all $x'\in
\Ob \FC$ homotopy equivalent to $x$) our (informal) approach to this problem
was that the "correct" derivation is the one giving the minimal results, i.e.,
if $\FC'\equiv \FC$ witnessed by the functor $F:\FC\lra \FC'$ which is
injective on $\Ob\FC$ then   the "correct" function to derive is $(f\circ F)$,
rather than $F$. The first problem for future research is, therefore
\begin{problem} Let $\FC$ be a posetal model category, $f:\FC\to On^T$ any
	function. Find a functor $\tilde {\L}_f f:\FC\lra On^{\top}$ such that
	\begin{enumerate} \item  ${\tilde \L}_f(x)\le \L_f(x)$ for all $x$.  \item
			$\L_f$ is not trivial (unless, say, $L_{(f\circ F)}$ is trivial for
			every functor $F:\FC'\to \FC$ with $\FC'\equiv \FC$.  \item
			$\tilde{\L}_f$	 is invariant under the equivalence of model categories
			(in the sense explained above).  \end{enumerate} In other words, extend
			the notion of the left derived functor of a functor $f:\FC\lra
			On^{\top}$ to a larger class of function with the result  as invariant
			as possible under equivalence of model categories.  \end{problem}

The second objection to the present work relates to its usefulness. Here is a
list of problems, a positive answer to some of which could indicate of the
usefulness of the new tools developed in the present work: 

\begin{problem} Are there more combinatorial concepts that can be captured by
	our suggested formalism, e.g., closed unbounded sets, stationary sets,
	Fodor's lemma, diamond, square etc..  \end{problem}

As a somewhat speculative special case of the previous problem consider the
fact that there are no measurable cardinals in $L$. As we have seen in Remark
\ref{nomeasurable} the statement "there are no measurable cardinals" can be
restated in our geometric language. Thus it is natural to ask: 

\begin{problem} Can it be proved using (mainly) the language of model
	categories that there are no measurable cardinals in $L$. In other words,
	can an analogue of Scott's theorem \cite{Scott} stating that if there are
	measurable cardinals then $V\neq L$ be proved using our geometric language?
\end{problem} As we do not have any "geometric" characterisation of $L$
(unlike, e.g., the set theoretic characterisation of $L$ being the smallest
inner model, i.e., the smallest submodel of $V$ containing all ordinals, or
the smallest transitive universal class closed under G\"odel operations), the
above question is somewhat speculative. As in our treatment of ordinals in
Section \ref{ordinals}, it seems reasonable to use some auxiliary notions such
as naming $On\in\Ob{\StNaamen}$ to address this problem. 

\begin{problem} Apparently, given a model structure on a category $\FC$, the
	computation of homotopy limits (i.e. the computation of limits in the
	associated homotopy category) gives in many cases important information on
	the category $\FC$. In the case of $\Qt$, one can easily give an explicit
	combinatorial interpretation of the limit (at least for set-sized diagrams).
	Are these objects of set theoretic significance? More generally, is there a
	set theoretic significance to the class of \emph{cute} objects?  To the
	homotopy category itself?  Are there other derived functors
	defining invariants of models of ZFC that, say, can be
	bounded in ZFC?  
\end{problem}

In classical homotopy theory, homotopy groups (by themselves) and the
associated structures (such as long exact sequences) are powerful tools
allowing many calculations. In (pointed) model categories  analogues of such
constructions exist, such as the groupoid of homotopy classes between any two
objects $A,B$ (where $A$ is a co-fibrant object and $B$ is a fibrant object)
as well as other constructs, analogous of other classical homotopical tools
such as the suspension and loop functors, fibration sequences and more. In
posetal model categories these constructions degenerate, and much of the
computational power of the associated homotopy structure is lost. This may be
one of the reasons that while we were able to recover homotopical
interpretations of important and non-trivial set theoretic objects we were
unable to prove any of their properties using the model category structure on
$\Qt$. 

In view of the above it is interesting to look for other constructions in
$\Qt$ (or $\rm{Ho}\Qt$), which may serve as analogues of the above mentioned
model categorical constructions. One possible such construction is the
sequence of model categories $\Qtk$ when $\kappa$ ranges over all cardinals. 

First, recall that we were able to give the category $\Qtk$ a model structure
only under the assumption that $\kappa$ is a regular cardinal. A first problem
is, therefore, to construct a similar model category for singular $\kappa$. It
seems that such a model category can be constructed inductively (assuming
$\Qt_{\lambda}$ was constructed for all $\lambda<\kappa$) by taking an
appropriate  "limiting" process. For  example, one could define \[
\Ob\qtk:=\left \{\bigcup_{\lambda<\kappa} X_{\lambda}: X_\lambda\in \Ob\qtk,
X_\lambda \subseteq X_{\lambda'} \text{ if } \lambda < \lambda'\right \} \]
with the additional requirement that if $X=\bigcup_{\lambda<\kappa} X_\lambda$
as above then the $X_{\lambda}$ are uniformly definable (this is required in
order to assure that $X$ is, indeed, a class). And the labelling \[
X\xrightarrow{(*)}_{\kappa} Y \iff (\forall^* \lambda)(y\in Y\Rightarrow
X\times \{y\} \xrightarrow{(*)}_{\lambda} Y \] where $\forall^* \lambda$ means
"for all large enough $\mu<\kappa$". Passing to the full sub-class of
\emph{cute} objects (with respect to this labelling) we apparently get a model
category $\Qtk$. It is unclear to us, however, whether this construction is
the "correct" one. 

There is also an obvious functor $F_{\kappa}: \qtk\lra \St_{\kappa^+}$ given
by $X\mapsto \bigcup_{<\kappa^+} X$ for $X\in \Ob\qtk$. Indeed, this is a
functor of model categories: this is obvious for (c) and (f) arrows, and not
much harder for (w)-arrows, with the conclusion following from Remark
\ref{intersect}.  On the level of the associated homotopy categories, it is
clear that, $\gamma(F_{\kappa}(X))=\perp_{\Qtk^+}$ for any co-fibrant $X\in
\Ob\Qtk$ (where, as above, $\gamma:\Qtk\lra \rm{Ho}\Qtk$ is the localization
functor).  Since the co-fibrant objects of any model category  suffice to
determine the associated homotopy category, it follows that the homotopy
category associated with the image of $\Qtk$ under $F_{\kappa}$ is trivial.
This gives the sequence of categories $\Qtk$ a certain flavour of "exactness",
which seem to require some further research.


\subsection{Looking back}We conclude these notes looking back to the original
motivation leading to the development of the model category $\Qt$, i.e., the
goal of developing a homotopy structure for the class of models of an
uncountably categorical theory and,  more generally, to (quasi-minimal)
excellent abstract elementary classes (see, e.g., \cite{Bal4} for the details).  
The need for homotopy theoretic tools in this contexts arose through the
study, by Zilber and his school, of categoricity problems of model theoretic
structures such as pseudo-exponentiation \cite{ZilCov} and covers of semi-Abelian
varieties \cite{BaysThesis}, \cite{BaZil}. 
The model theoretic analysis needed to show  the (uncountable) categoricity of the natural examples studied in the above mentioned references 
uses known number theoretic and algebro-geometric results and conjectures 
nowadays understood as being of essentially
cohomological character, and formulated in functorial language. Such statements are,  e.g.,
particular cases of Andr\'e's generalized Grothendieck conjectures on periods
of motives (\cite{Bertolin},
\cite[7.5.2.1 Conjecture]{Andre},\cite[\S4.2 Conjecture,\S1.2 Conjecture]{KonZag}),
the Mumford-Tate conjecture on the
image of Galois action on the first \'etale cohomology (\cite{Serre}), Kummer
theory (\cite{Rib}), and more. This does not seem to be entirely coincidental, as  the definition of Shelah's, so called, Excellent classes - the model theoretic machinery employed in this study -  and in particular the requirement that there exists a 
unique prime model over maximally independent tuples of
countable (sub) models (and that this requirement makes sense) reminds, 
at least superficially, some of the axioms of a model category.

However, the common model theoretic language does not seem to have the means to incorporate 
this functorial language in its full power and generality. Thus, in order to be applied in 
addressing the above mentioned categoricity problems "old-fashioned" reformulations of 
these conjectures, deprived of their functorial language and homological character 
had to be used --- Schanuel's conjecture and its cognates explicated by Bertolin \cite{Bertolin} derived 
from the generalized Grothendieck conjecture on periods, Bashmakov's
original formulations of Kummer theory for elliptic curves (\cite{BaysThesis},
\cite{GavK}), and Serre's explicit description of the image of the Galois action 
on the Tate module as a subgroup of the profinite group $\GL_2 (\hat {\mathbb Z})$.

It the first author's belief that the inability of common model theoretic
language to digest these statements in their full power and generality is a
major obstacle in further exploring  these intriguing connections between
Shelah's excellent classes and deep algebro-geometric conjectures. The
homotopy theoretic approach to set theory discussed in the present paper and
in \cite{GaHa} is a toy example exploring the ways in which  homotopy
theoretic language could be introduced into the realm of model theory. 

Unfortunately, we were unable to use the model category $\Qt$ to associate such a homotopy structure to those classes of models. In fact, it is not even clear to us when this could be done: 

\begin{problem}\label{problem:aec}
Let $\mathfrak K$ be a (quasiminimal) excellent abstract elementary class (e.g. algebraically closed fields of characteristic $p$, models of pseudo-exponentiation). Let $\Qt(\mathfrak K)$ be the sub-category of $\Qt$ whose objects are elements of $\mathfrak K$ and such that for $\mathcal M, \mathcal N\in \Ob \mathfrak K$ there is an arrow $\mathcal M\lra \mathcal N$ if $\mathcal M\prec \mathcal N$. Are there natural model theoretic conditions under which $\Qt(\mathfrak K)$ is a model category? What about $\Qtk(\mathfrak K)$? Is there a similiar construction 
associating a model category to the class $\mathfrak K$?
\end{problem}

\noindent {\bf Acknowledgments.} 
The first author wishes to  thank his  Mother and Father for support, patience and more,  Boris Zilber to whose ideas this work owes a large debt, Artem Harmaty for attention to this work, and 
encouraging conversations.
 The first author would also like to thank the St. Petersburg Steklov Mathematical Institute (PDMI RAS) 
 and the participants of the seminar organised by N. Durov and A. Smirnov for their hospitality and insight into model categories.

The authors would like to thank Martin Bays, Sharon Hollander, Marco Porta, Alex Usvyatsov  and Mike Shulman for
reading very early drafts of this works, for their comments, suggestions and corrections. Thanks are also due to Sy Fridman and to Lyubomyr Zdomskyy for their comments, questions and ideas, some of which made their way into the present paper . 

%

\bibliographystyle{plain}
\bibliography{../../Bibfiles/harvard}

\begin{thebibliography}{10}

\bibitem{Andre}
Yves Andr{\'e}.
\newblock {\em Une introduction aux motifs (motifs purs, motifs mixtes,
  p{\'e}riodes)}, volume~17 of {\em Panoramas et Synth{\`e}ses [Panoramas and
  Syntheses]}.
\newblock Soci{\'e}t{\'e} Math{\'e}matique de France, Paris, 2004.

\bibitem{Bal4}
John~T. Baldwin.
\newblock {\em Categoricity}, volume~50 of {\em University Lecture Series}.
\newblock American Mathematical Society, Providence, RI, 2009.

\bibitem{BaysThesis}
Martin Bays.
\newblock {\em Categoricity Results for Exponential Maps of 1-Dimensional
  Algebraic Groups \& Schanuel Conjectures for Powers and the CIT}.
\newblock PhD thesis, Oxford University, 2009.
\newblock Available at http://people.maths.ox.ac.uk/bays/dist/thesis/.

\bibitem{BaZil}
Martin Bays and Boris Zilber.
\newblock Covers of multiplicative groups of algebraically closed fields of
  arbitrary characteristic.
\newblock {\em Bull. Lond. Math. Soc.}, 43(4):689--702, 2011.

\bibitem{Bertolin}
Cristiana Bertolin.
\newblock P{\'e}riodes de 1-motifs et transcendance.
\newblock {\em J. Number Theory}, 97(2):204--221, 2002.

\bibitem{GaHa}
M.~Gavrilovich and Assaf Hasson.
\newblock Exercises de style: a homotopy theory for set theory {I}.
\newblock Available on arXiv, 2010.

\bibitem{GavK}
Misha Gavrilovich.
\newblock A remark on transitivity of {G}alois action on the set of uniquely
  divisible abelian extensions in {${\rm Ext}^1(E(\overline{\Bbb
  Q}),\Lambda)$}.
\newblock {\em -Theory$K$}, 38(2):135--152, 2008.

\bibitem{KojABC}
Menachem Kojman.
\newblock The {A,B,C} of {PCF}.
\newblock available on http://de.arxiv.org/pdf/math/9512201.pdf, 1995.

\bibitem{KojmanHistory}
Menachem Kojman.
\newblock Singular cardinals: from {H}ausdorff's gaps to {S}helah's {PCF}
  theory.
\newblock In Akihiro Kanamori, editor, {\em The Handbook of the History of
  Logic (Vol. 6), {S}ets and extensions in the twentieth century}. Elsevier,
  2012.
\newblock In production.

\bibitem{KonZag}
Maxim Kontsevich and Don Zagier.
\newblock Periods.
\newblock In {\em Mathematics unlimited---2001 and beyond}, pages 771--808.
  Springer, Berlin, 2001.

\bibitem{Rib}
Kenneth~A. Ribet.
\newblock Kummer theory on extensions of abelian varieties by tori.
\newblock {\em Duke Math. J.}, 46(4):745--761, 1979.

\bibitem{Scott}
Dana Scott.
\newblock Measurable cardinals and constructible sets.
\newblock {\em Bull. Acad. Polon. Sci. S{\'e}r. Sci. Math. Astronom. Phys.},
  9:521--524, 1961.

\bibitem{Serre}
Jean-Pierre Serre.
\newblock Propri{\'e}t{\'e}s conjecturales des groupes de {G}alois motiviques
  et des repr{\'e}sentations {$l$}-adiques.
\newblock In {\em Motives ({S}eattle, {WA}, 1991)}, volume~55 of {\em Proc.
  Sympos. Pure Math.}, pages 377--400. Amer. Math. Soc., Providence, RI, 1994.

\bibitem{ShCard}
Saharon Shelah.
\newblock {\em Cardinal arithmetic}, volume~29 of {\em Oxford Logic Guides}.
\newblock The Clarendon Press Oxford University Press, New York, 1994.
\newblock Oxford Science Publications.

\bibitem{ZilCov}
Boris Zilber.
\newblock Covers of the multiplicative group of an algebraically closed field
  of characteristic zero.
\newblock {\em J. London Math. Soc. (2)}, 74(1):41--58, 2006.

\end{thebibliography}

\end{document}